\documentclass[11pt,reqno]{amsart}
\usepackage{amsmath, amssymb, amsthm, graphicx}
\usepackage{stmaryrd}
\usepackage{tikz}
\usetikzlibrary{arrows,calc,shapes}
\font\dixmath=cmsy10
\def\emptyset{\varnothing}
\def\leq{\leqslant}
\def\geq{\geqslant}
\newcommand{\ssstate}[3]{\node[draw,circle,inner sep=0,minimum size=24,line width=0.8] (#1) at #2 {$\scriptstyle#3$};}
\newcommand{\sstate}[3]{\node[draw,circle,inner sep=0,minimum size=18,line width=0.8] (#1) at #2 {$#3$};}
\newcommand{\state}[4]{\sstate{#1}{#2}{\substack{#3\\#4}}}

\newcommand{\myloop}[4]{\path[->,>=latex,line width=0.8] (#1) edge [out=#3+25,in=#3-25,distance=20] node [label={[label distance=#4]#3:$\scriptstyle#2$}] {} (#1);}
\newcommand{\myloops}[5]{\path[->,>=latex,line width=0.8] (#1) edge [out=#3+25,in=#3-25,distance=20] (#1); \draw[-] (#4,#5) node{$\scriptstyle #2$};}

\newcommand{\myloopg}[4]{\path[color=gray,->,>=latex,line width=0.8] (#1) edge [out=#3+25,in=#3-25,distance=20] node [label={[label distance=#4]#3:$\scriptstyle#2$}] {} (#1);}
\newcommand{\initstate}[2]{\draw[<-,>=latex,line width=0.8] (#1.center) ++ (#2:9pt) --++ (#2:20pt);}
\newcommand{\initstatea}[3]{\draw[<-,>=latex,line width=0.8] (#1.center) ++ (#2:9pt) --++ (#2:20pt) node[above]{$\scriptstyle#3$};} 
\newcommand{\initstateb}[3]{\draw[<-,>=latex,line width=0.8] (#1.center) ++ (#2:9pt) --++ (#2:20pt) node[below]{$\scriptstyle#3$};} 
\newcommand{\initstatel}[3]{\draw[<-,>=latex,line width=0.8] (#1.center) ++ (#2:9pt) --++ (#2:20pt) node[left]{$\scriptstyle#3$};} 
\newcommand{\initstateld}[3]{\draw[<-,>=latex,line width=0.8,dashed] (#1.center) ++ (#2:9pt) --++ (#2:20pt) node[left]{$\scriptstyle#3$};} 
\newcommand{\initstatelg}[3]{\draw[<-,>=latex,line width=0.8,color=gray] (#1.center) ++ (#2:9pt) --++ (#2:20pt) node[left]{$\scriptstyle#3$};}

\renewcommand{\aa}[2]{a_{#1,#2}}
\newcommand{\aaa}[2]{a_{#1\hspace{-0.8pt},\hspace{-0.4pt}#2}}

\newcommand{\BB}[1]{B_{#1}}
\newcommand{\bidual}{^{\hspace{-0.05em}\raise0.6pt\hbox{$\scriptscriptstyle+$}\hspace{-0.05em}*}}
\newcommand{\bidualr}{^{*\hspace{-0.05em}\raise0.6pt\hbox{$\scriptscriptstyle+$}\hspace{-0.05em}}}
\newcommand{\BKL}[1]{B_{#1}\bidual}
\newcommand{\SPBKL}[1]{A_{#1}}
\newcommand{\SBKL}[1]{A_{#1}^\pm}
\newcommand{\BP}[1]{B_{#1}\plusexp}

\renewcommand{\bar}[1]{\text{bar}(#1)}
\newcommand{\br}{\beta}
\newcommand{\brbr}{\gamma}

\newcommand{\brr}{\beta'}

\newcommand{\equal}\equiv
\newcommand{\ff}[1]{\phi_{\hspace{-0.5pt}\raise-1pt\hbox{$\scriptstyle #1$}}}

\newcommand{\fl}[1]{\Phi_{\hspace{-1pt}\raise-0.5pt\hbox{$\scriptstyle #1$}}}

\renewcommand{\ge}{\geqslant}

\newcommand{\ie}{\emph{i.e.}}

\newcommand{\inv}{^{\minus\hspace{-0.1em}1}}

\newcommand{\last}[1]{{#1}^{\scriptscriptstyle\mathtt{\#}}}
\newcommand{\lcmL}{\vee_{\hspace{-0.2em}\raise1pt\hbox{\(\scriptscriptstyle L\)}}}
\newcommand{\lcmR}{\vee_{\hspace{-0.2em}\raise1pt\hbox{\(\scriptscriptstyle R\)}}}
\newcommand{\Ldots}{...\,} 
\renewcommand{\le}{\leqslant}

\renewcommand{\mod}{\ \text{mod}\ }
\newcommand{\minus}{\mathchoice{-}{-}{\raise0.7pt\hbox{$\scriptscriptstyle-$}\scriptstyle}{-}}
\newcommand{\NN}{\mathbb{N}}
\newcommand{\newintegeri}[2]{
  \expandafter\def\csname #1\endcsname{#2}
  \expandafter\def\csname #1o\endcsname{{#2\minus1}}
  \expandafter\def\csname #1t\endcsname{{#2\minus2}}
  \expandafter\def\csname #1p\endcsname{{#2\plus1}}
  \expandafter\def\csname #1pp\endcsname{{#2\plus2}}}
  
\newcommand{\newintegerii}[1]{\newintegeri{#1#1}{#1}}

\newcommand{\plus}{\mathchoice{+}{+}{\raise0.7pt\hbox{$\scriptscriptstyle+$}\scriptstyle}{+}}
\newcommand{\plusminus}{\mathchoice{\pm}{\pm}{\raise0.7pt\hbox{$\scriptscriptstyle\pm$}\scriptstyle}{\pm}}
\newcommand{\plusexp}{^{\raise0.8pt\hbox{$\hspace{-0.05em}\scriptscriptstyle+$}}}
\newcommand{\rev}{\curvearrowright}
\newcommand{\rnf}{r}
\newcommand{\sh}[1]{#1^{\hspace{-0.2pt}\raise0.6pt\hbox{$\scriptscriptstyle+$}}}
\newcommand{\sig}[1]{\sigma_{\!#1}} 
\newcommand{\siginv}[1]{\sigma_{\!#1}^{\hspace{-0.05em}\raise0.8pt\hbox{$\scriptscriptstyle-$}\hspace{-0.1em}1}}
\newcommand{\sigpm}[1]{\sigma_{\!#1}^{\hspace{-0.05em}\raise0.8pt\hbox{$\scriptscriptstyle\pm$}\hspace{-0.1em}1}}

\newcommand{\uu}{u}

\newcommand{\vv}{v}

\newcommand{\ww}{w}

\newcommand{\www}{\ww'}

\newcommand{\xx}{x}

\newcommand{\yy}{y}

\newintegeri{indi}{p}
\newintegeri{indii}{q}
\newintegeri{indiii}{r}
\newintegeri{indiv}{s}
\newintegeri{indv}{t}
\newintegeri{indvi}{t'}
\newintegeri{brdi}{b}
\newintegeri{brdii}{c}
\newintegeri{brdiii}{t}
\newintegeri{brdiv}{t'}
\newintegeri{ll}{\ell}
\newintegerii{d}
\newintegerii{e}
\newintegerii{h}
\newintegerii{i}
\newintegerii{j}
\newintegerii{k}
\newintegerii{m}
\newintegerii{n}
\newintegerii{p}
\newintegerii{q}
\newintegerii{r}
\newintegerii{s}
\newintegerii{t}

\newcommand{\RNF}[1]{R_{#1}}

\title{The rotating normal form is regular}
\author{Jean Fromentin}

\theoremstyle{definition}

\newtheorem{defi}{Definition}[section]

\newtheorem{exam}[defi]{Example}

\theoremstyle{plain}

\newtheorem{lemm}[defi]{Lemma}
\newtheorem{prop}[defi]{Proposition}

\newtheorem{thrm}[defi]{Theorem}
\newtheorem{coro}[defi]{Corollary}

\begin{document}
\maketitle

\begin{abstract}
  Defined on  Birman--Ko--Lee monoids, the rotating normal form has strong 
  connections with the Dehornoy's braid ordering.
  It can be seen as a process for selecting between all the 
  representative words of a Birman--Ko--Lee braid a particular one, called \emph{rotating} word.
  In this paper we construct, for all $n\geq 2$, a finite state automaton which recognizes
  the rotating words on $n$ strands.
  As a consequence the language of rotating words on $n$ strands is proved 
  to be regular for any $n\geq 2$.
\end{abstract}

\section{Introduction}
Originally, the group~$\BB\nn$ of $\nn$-strand braids was defined as the group 
of isotopy classes of $\nn$-strand geometric braids.
An algebraic presentation of~$\BB\nn$ was given by E.~Artin in \cite{Artin1925}
\begin{equation}
  \label{E:BnPresentation}
  \left<\sig1,\Ldots,\sig\nno \left| 
      \begin{array}{cl}
        \sig\ii\sig\jj\,=\,\sig\jj\sig\ii & \text{for $|\ii\minus\jj|\ge2$}\\
        \sig\ii\sig\jj\sig\ii\,=\,\sig\jj\sig\ii\sig\jj & \text{for $|\ii\minus\jj|=1$} 
      \end{array}\right.\right>.
\end{equation}
An $\nn$-strand braid is an equivalence class consisting of (infinitely many)
words in the letters $\sig\ii^{\pm1}$. 
The standard correspondence between elements of the presented group $\BB\nn$ and 
geometric braids consists in using $\sig\ii$ as a code for the geometric braid where only
the $\ii$th and the $(\iip)$st strands cross, with the strands originally at position
$(\iip)$ in front of the other.

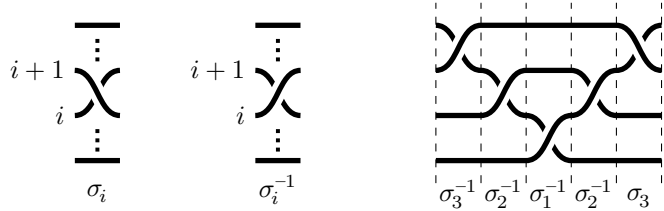
\begin{figure}[ht!]
  \begin{center}
    \begin{tikzpicture}[x=0.06cm,y=0.06cm]
      \draw[line width=2](0,0) -- (10,0);
      \draw[line width=2](0,10) .. controls (5,10) and (5,20) .. (10,20);
      \draw[line width=2](0,20) .. controls (5,20) and (5,10) .. (10,10);
   \draw[line width=6,color=white](0,20) .. controls (5,20) and (5,10) .. (10,10);
         \draw[line width=2](0,20) .. controls (5,20) and (5,10) .. (10,10);
      \draw[line width=2](0,30) -- (10,30);
      \draw[line width=1.5,dotted](5,2) -- (5,8); 
      \draw[line width=1.5,dotted](5,22) -- (5,28); 
      \draw(0,10) node[left]{\small$\ii$};
      \draw(0,20) node[left]{\small$\iip$};
      \draw(5,-3) node[below]{\small$\sig\ii$};
      \begin{scope}[shift={(40,0)}]
        \draw[line width=2](0,0) -- (10,0);
        \draw[line width=2](0,10) .. controls (5,10) and (5,20) .. (10,20);
        \draw[line width=2](0,20) .. controls (5,20) and (5,10) .. (10,10);
        \draw[line width=6,color=white](0,10) .. controls (5,10) and (5,20) .. (10,20);
        \draw[line width=2](0,10) .. controls (5,10) and (5,20) .. (10,20);
        \draw[line width=2](0,30) -- (10,30);
        \draw[line width=1.5,dotted](5,2) -- (5,8); 
        \draw[line width=1.5,dotted](5,22) -- (5,28); 
        \draw(0,10) node[left]{\small$\ii$};
        \draw(0,20) node[left]{\small$\iip$};
        \draw(5,-1) node[below]{\small$\siginv\ii$};
      \end{scope}
      \begin{scope}[shift={(80,0)}]
        \draw[line width=2](0,30) .. controls (5,30) and (5,20) .. (10,20) .. controls (15,20) and (15,10) .. (20,10) .. controls (25,10) and (25,0) .. (30,0) -- (50,0);
        \draw[line width=6,color=white](0,10) -- (10,10) .. controls (15,10) and (15,20) .. (20,20) -- (30,20) .. controls (35,20) and (35,10) .. (40,10) -- (50,10);
        \draw[line width=2](0,10) -- (10,10) .. controls (15,10) and (15,20) .. (20,20) -- (30,20) .. controls (35,20) and (35,10) .. (40,10) -- (50,10);
        \draw[line width=6,color=white] (0,0) -- (20,0) .. controls (25,0) and (25,10) .. (30,10) .. controls (35,10) and (35,20) .. (40,20) .. controls (45,20) and (45,30) .. (50,30);
        \draw[line width=2] (0,0) -- (20,0) .. controls (25,0) and (25,10) .. (30,10) .. controls (35,10) and (35,20) .. (40,20) .. controls (45,20) and (45,30) .. (50,30);
        \draw[line width=6,color=white] (0,20) .. controls (5,20) and (5,30) .. (10,30) -- (40,30) .. controls (45,30) and (45,20) .. (50,20);
        \draw[line width=2] (0,20) .. controls (5,20) and (5,30) .. (10,30) -- (40,30) .. controls (45,30) and (45,20) .. (50,20);
        \draw[dashed](0,-5) -- (0,35);
        \draw[dashed](10,-5) -- (10,35);
        \draw[dashed](20,-5) -- (20,35);
        \draw[dashed](30,-5) -- (30,35);
        \draw[dashed](40,-5) -- (40,35);
        \draw[dashed](50,-5) -- (50,35);
        \draw(5,-2) node[below]{\small$\siginv3$};
        \draw(15,-2) node[below]{\small$\siginv2$};
        \draw(25,-2) node[below]{\small$\siginv1$};
        \draw(35,-2) node[below]{\small$\siginv2$};
        \draw(45,-4) node[below]{\small$\sig3$};
      \end{scope}
    \end{tikzpicture}
    
  \end{center}
  \caption{Interpretation of a word in the letters $\sigpm\ii$ as a geometric braid diagram.}
\end{figure}

In 1998, J.S. Birman, K.H. Ko, and S.J. Lee~\cite{Birman1998} introduced and 
investigated for each~$\nn$ a submonoid~$\BKL\nn$ of~$\BB\nn$, which is known as the \emph{Birman--Ko--Lee} monoid. 
The name~``\emph{dual braid monoid}'' was subsequently proposed because
several numerical parameters obtain symmetric values when they are evaluated on the 
positive braid monoid~$\BP\nn$ and on~$\BKL\nn$, a correspondence that was extended to the more general context of
Artin--Tits groups by D. Bessis~\cite{Bessis2003} in 2003.
The dual braid monoid~$\BKL\nn$ is the submonoid of~$\BB\nn$ generated by 
the braids~$\aa\ii\jj$ with $1\le\ii<\jj\le\nn$, where $\aa\ii\jj$ is defined by 
$\aa\ii\jj=\sig\ii\,...\,\sig\jjo\ \sig\jj\ \siginv\jjo\,...\,\siginv\ii$. 
In geometrical terms, the braid~$\aa\ii\jj$ corresponds to a crossing of the
$\ii$th and $\jj$th strands, both passing behind the (possible) intermediate strands. 

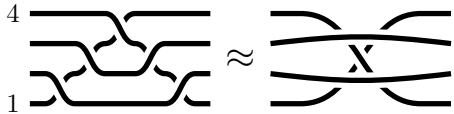
\begin{figure}[ht!]
  \begin{center}
    \begin{tikzpicture}[x=0.04cm,y=0.04cm]
      \draw[line width=2] (-5,0) -- (0,0) .. controls (5,0) and (5,10) .. (10,10) .. controls (15,10) and (15,20) .. (20,20) .. controls (25,20) and (25,30) .. (30,30) -- (55,30); 
      \draw[line width=6,color=white] (-5,30) -- (20,30) .. controls (25,30) and (25,20) .. (30,20) .. controls (35,20) and (35,10) .. (40,10) .. controls (45,10) and (45,0) .. (50,0) -- (55,0); 
      \draw[line width=2] (-5,30) -- (20,30) .. controls (25,30) and (25,20) .. (30,20) .. controls (35,20) and (35,10) .. (40,10) .. controls (45,10) and (45,0) .. (50,0) -- (55,0); 
      \draw[line width=6,color=white] (-5,10) -- (0,10) .. controls (5,10) and (5,0) .. (10,0) -- (40,0) .. controls (45,0) and (45,10) .. (50,10) -- (55,10);
      \draw[line width=2] (-5,10) -- (0,10) .. controls (5,10) and (5,0) .. (10,0) -- (40,0) .. controls (45,0) and (45,10) .. (50,10) -- (55,10);
      \draw[line width=6,color=white] (-5,20) -- (10,20) .. controls (15,20) and (15,10) .. (20,10) -- (30,10) .. controls (35,10) and (35,20) .. (40,20) -- (55,20);
      \draw[line width=2] (-5,20) -- (10,20) .. controls (15,20) and (15,10) .. (20,10) -- (30,10) .. controls (35,10) and (35,20) .. (40,20) -- (55,20);
      \draw(-5,0) node[left]{\small$1$};
      \draw(-5,30) node[left]{\small$4$};
      \draw(65,15) node{\Large$\approx$};
      \begin{scope}[shift={(75,0)}]
        \draw[line width=2](0,0) -- (10,0) .. controls (30,0) and (30,30) .. (50,30) -- (60,30); 
        \draw[line width=6,color=white](0,30) -- (10,30) .. controls (30,30) and (30,0) .. (50,0) -- (60,0); 
        \draw[line width=2](0,30) -- (10,30) .. controls (30,30) and (30,0) .. (50,0) -- (60,0); 
        \draw[line width=6,color=white](0,10) .. controls (25,7) and (35,7) .. (60,10);
        \draw[line width=2](0,10) .. controls (25,7) and (35,7) .. (60,10);
        \draw[line width=6,color=white](0,20) .. controls (25,23) and (35,23) ..  (60,20);
        \draw[line width=2](0,20)  .. controls (25,23) and (35,23) ..  (60,20);
      \end{scope}
    \end{tikzpicture}
  \end{center}
  \caption{In the geometric braid \(\aa14\), the strands  \(1\) and \(4\) cross under the strands  \(2\) and \(3\).}
\end{figure}

By definition, $\sig\ii$ equals $\aa\ii\iip$ and, therefore, the positive braid 
monoid~$\BP\nn$ is included in the monoid~$\BKL\nn$, a proper inclusion 
for~$\nn \ge 3$ since the braid $\aa13$ does not belong to the monoid~$\BP3$.

We denote by $\SPBKL\nn$  the set $\{\aa\indi\indii\,|\,1\leq\indi<\indii\leq\nn\}$.
The following presentation of the monoid $\BKL\nn$ is given in \cite{Birman1998}.

\begin{prop}
  \label{P:PresBKL}
 The monoid~$\BKL\nn$ is presented by generators $\SPBKL\nn$ and relations
  \begin{align}
    \label{E:DualCommutativeRelation}
    \aa\indi\indii\aa\indiii\indiv&= \aa\indiii\indiv\aa\indi\indii \text{\quad for $[\indi, \indii]$ and $[\indiii, \indiv]$ disjoint or nested},\\
    \label{E:DualNonCommutativeRelation}
    \aa\indi\indii\aa\indii\indiii&= \aa\indii\indiii\aa\indi\indiii =\aa\indi\indiii\aa\indi\indii \text{\quad for $1 \le \indi<\indii<\indiii\le\nn$}.
  \end{align}
\end{prop}
The integral interval $[p,q]$ is said to be \emph{nested} in $[r,s]$ if the relation $r<p<q<s$ holds.

Since~\cite{Bessis2003} and \cite{Birman1998} it is known that the dual braid monoid~$\BKL\nn$ admits a 
Garside structure whose simple elements are in bijection with the non-crossing 
partitions of $\nn$.
In particular, there exists a normal form associated with this Garside
structure, the so-called greedy normal form.

The rotating normal form is another  normal form on $\BKL\nn$, it was
introduced in \cite{Fromentin2008,Fromentin2008a}.
Roughly speaking, for every braid $\br\in\BKL\nn$ the rotating normal form picks up a 
unique representative word on the letters~$\SPBKL\nn$ among all of these representing~$\br$. 
It can be see as a map $r_n$ from the dual braid monoid $\BKL\nn$ 
to the set of words~$\SPBKL\nn^\ast$.
The language of all $n$-rotating words, denoted by~$R_n$ is then 
the image of~$\BKL\nn$ under the map $r_n$.

The aim of this paper is to construct for all $n\geq 2$ an explicit finite state 
automaton which recognizes the language~$\RNF\nn$.
As a consequence we obtain that the language~$\RNF\nn$ of $n$-rotating words is regular.

The paper is divided as follow.
In section 2 we recall briefly the construction of the rotating normal form
and its useful already known properties. 
In third section we describe the left reversing process on dual braid monoids.
In section~$4$ we give a syntactical characterization of $n$-rotating normal words.
In fifth section we construct, for each $n\geq 2$, a finite state automaton which
recognizes the language $\RNF\nn$ of $n$-rotating normal words.

\section{The rotating normal form}

The main ingredient to define the rotating normal form is the Garside automorphism 
$\ff\nn$ of $\BKL\nn$ defined by $\ff\nn(\br)=\delta_\nn\,\br\,\delta_\nn\inv$ where
$\delta_\nn=\aa12\,\aa23\,...\aa\nno\nn$ is the Garside braid of $\BKL\nn$.
In terms of Birman--Ko--Lee generators, the map $\ff\nn$ can be defined by 
\begin{equation}
  \label{E:Phi}
  \ff\nn(\aa\indi\indii)=\begin{cases}
    \aa\indip\indiip&\text{for $\indii\le\nno$,}\\
    \aa1\indip&\text{for $\indii=\nn$}.
  \end{cases} 
\end{equation}
Geometrically, $\ff\nn$ should be viewed as a rotation, which makes sense provided 
braid diagrams are drawn on a cylinder rather than on a plane rectangle.

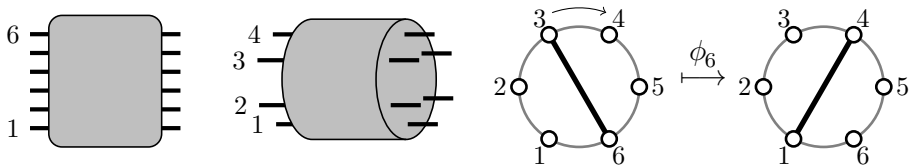
\begin{figure}[h!]
  \begin{center}
  \end{center}
  \begin{tikzpicture}[x=0.05cm,y=0.05cm]
    \foreach \v in {1,2,3,4,5,6}{\draw[line width=1.6](-5,5*\v) -- (35,5*\v);}
    \draw(-5,5) node[left]{\small$1$};
    \draw(-5,30) node[left]{\small$6$};
    \draw[rounded corners=5,fill=lightgray,line width=0.6] (0,0) -- (30,0) -- (30,35) -- (0,35) -- cycle;
    \begin{scope}[shift={(70,2)}]
      \foreach \v in {1,2,3,4,5,6}{\draw[line width=1.6](-10,16) ++ (\v*60+35:5 and 12) -- ++ (8,0);}
      \draw[fill=lightgray,line width=0.6](0,0) -- (25,0)  arc(-90:90:8 and 16) -- ++ (-25,0) arc(90:270:8 and 16);
      \draw[line width=0.6](25,0) arc (270:90:8 and 16);
      \foreach \v in {1,2,3,4,5,6}{\draw[line width=1.6](25,16) ++ (-\v*60-25:5 and 12) -- ++ (8,0);}
      \foreach \v in {1,2,3,4}{\draw[line width=1.6](-10,16) ++ (-\v*60-25:5 and 12) node[left]{\small$\v$};}
    \end{scope}
    \begin{scope}[shift={(125,0)}]
      \draw[line width=1,color=gray](16,16) circle (16);
      \draw[line width=2] (16,16) + (120:16) -- +(300:16);
      \draw[line width=2] (16,16) + (120:16) -- +(180:16);
      \draw[line width=2] (16,16) + (180:16) -- +(300:16);
      \draw[line width=2] (16,16) + (60:16) -- +(0:16);
      \draw[<-](16,16) + (130:21) arc(130:170:21);
      \foreach \v in {1,2,3,4,5,6}{
        \draw[fill=white,line width=1](16,16) ++ (-\v*60-60:16) circle (2) ++ (-\v*60-60:5) node{\small$\v$};
      }
      \begin{scope}[shift={(65,0)}]
        \draw[line width=1,color=gray](16,16) circle (16);
        \draw[line width=2] (16,16) + (60:16) -- +(-120:16);
        \draw[line width=2] (16,16) + (60:16) -- +(120:16);
        \draw[line width=2] (16,16) + (120:16) -- +(-120:16);
        \draw[line width=2] (16,16) + (0:16) -- +(-60:16);
        \foreach \v in {1,2,3,4,5,6}{
          \draw[fill=white,line width=1](16,16) ++ (-\v*60-60:16) circle (2) ++ (-\v*60-60:5) node{\small$\v$};
        }
      \end{scope}
      \draw(49,18) node{$\mathbb{\longmapsto}$};
      \draw(49,25) node{$\mathbb{\ff6}$};  
    \end{scope}

  \end{tikzpicture}

  \caption{Rolling up the usual braid diagram helps us to visualize the symmetries of the braids $\aa\indi\indii$.
    On the resulting cylinder, $\aa\indi\indii$ naturally corresponds to the chord
    connecting vertices $\indi$ and $\indii$. With this representation, $\ff\nn$ acts
    as a  clockwise rotation of the marked circles by $2\pi/n$.
  }
  \label{F:CylinderA}
\end{figure}

For $\br$ and $\brbr$ in $\BKL\nn$, we say that $\brbr$ is a \emph{right-divisor} of $\br$, if there exists a dual braid~$\brr$ of $\BKL\nn$ satisfying $\br=\brr\,\brbr$.

\begin{defi}
 For $n\geq 3$ and $\br$ a braid of $\BKL\nn$, the maximal braid $\br_1$ lying in $\BKL\nno$ that 
 right divides the braid $\br$ is called the \emph{$\BKL\nno$-tail} of $\br$.
 \end{defi}

Using basic Garside properties of the monoid~$\BKL\nn$ we obtain the following result (Proposition 2.5 of \cite{Fromentin2008a}) which allow us to express each braid of $\BKL\nn$ as a unique finite sequence of braids lying in $\BKL\nno$.

\begin{prop}
\label{P:Splitting}
Assume $\nn\geq 3$. For each nontrivial braid $\br$ of $\BKL\nn$ 
there exists a unique sequence $(\br_\brdi,\Ldots,\br_1)$ of braids of $\BKL\nno$ satisfying $\br_\brdi\not=1$ and
\begin{gather}
 \br=\ff\nn^\brdio(\br_\brdi)\cdot...\cdot\ff\nn(\br_2)\cdot\br_1,\\
\label{E2:P:Split} \text{for each $\kk\geq 1$, the $\BKL\nno$-tail of $\ff\nn^{\brdi-\kk}(\br_\brdi)\cdot...\cdot\ff\nn(\br_\kkp)$ is trivial.}
\end{gather}
\end{prop}

Under the above hypotheses,\! the sequence\! $(\br_\brdi,\Ldots,\br_1)$ is called the 
\emph{$\ff\nn$-splitting} of the braid~$\br$.
It is shown in \cite{Fromentin2008a} that Condition~\eqref{E2:P:Split} can be replaced by
\begin{equation}
 \label{E2:P:Split:v2}
\text{for each $\kk\leq1$, $\br_\kk$ is the $\BKL\nno$-tail of $\ff\nn^{\brdi-\kk}(\br_\brdi)\cdot...\cdot\ff\nn(\br_\kko)\cdot\br_\kk$.}
\end{equation}

\begin{figure}[htb]
  \begin{center}
    \begin{tikzpicture}[x=0.05cm,y=0.05cm]
      \draw(0,0) node{\includegraphics{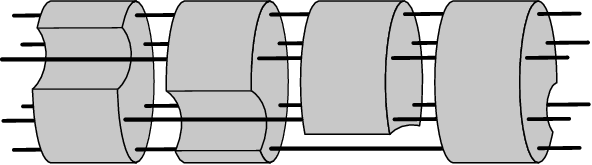}};
      \draw(100,-22) node{$\tiny 1$};
      \draw(96,-12) node{$\tiny 2$};
      \draw(96,8) node{$\tiny 3$};
      \draw(100,24) node{$\tiny 4$};
      \draw(103,14) node{$\tiny 5$};
      \draw(103,-7) node{$\tiny 6$};
      \draw(63,0) node{$\br_1$};
      \draw(16,5) node{$\ff6(\br_2)$};
      \draw(-29,10) node{$\ff6^2(\br_3)$};
      \draw(-73,-13) node{$\ff6^3(\br_4)$};
    \end{tikzpicture}
  \end{center}
  \caption{The \(\ff6\)-splitting of a braid of~\(\BKL6\).
    Starting from the right, we extract the maximal right-divisor that keeps the sixth strand unbraided, then extract the maximal right-divisor that keeps the first strand unbraided, etc.}
\end{figure}
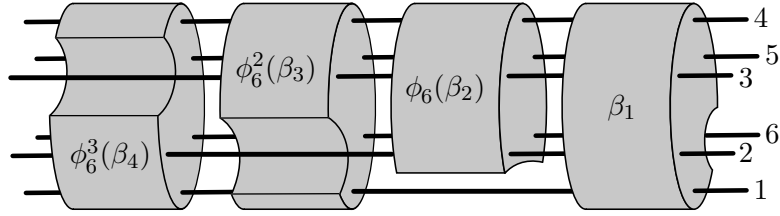

\begin{exam}
\label{E:Split}
  Consider the braid $\br=\aa12\aa23\aa12\aa23$ of $\BKL3$.
  Using relations \eqref{E:DualNonCommutativeRelation} on the underlined factors we obtain
  \[
  \br=\aa12\aa23\underline{\aa12\aa23}=\aa12\underline{\aa23\aa13}\aa12=\aa12\aa13\aa12\aa12
  \]
We decompose $\br$ as $\ff3(\brbr_1)\cdot\br_1$ with $\brbr_1=\ff3\inv(\aa12\aa13)=\aa13\aa23$ and $\br_1=\aa12\aa12$.
The braid $\ff3(\brbr_1)=\aa12\aa23$ is exactly the one of \eqref{E2:P:Split} for $\nn=3$ and $\kk=1$.
As the word $\aa13\aa23$ is alone in its equivalence class the braid $\ff3(\brbr_1)$ is not right divisible by $\aa12$ and so its $\BKL2$-tail is trivial. 
Considering $\brbr_1$ instead of $\br$ we obtain $\brbr_1=\ff3(\brbr_2)\cdot \br_2$ with $\brbr_2=\ff3\inv(\aa13\aa23)=\aa23\aa12$ and $\br_2=1$.
The braid $\ff3(\brbr_2)=\aa13\aa23$  is the braid of \eqref{E2:P:Split} for $\nn=3$ and $\kk=2$ and it is always alone in its equivalence class, implying that its $\BKL2$-tail is trivial.
We express $\brbr_2$ as $\ff3(\brbr_3)\cdot\br_3$ with $\brbr_3=\ff3\inv(\aa23)=\aa12$ and $\br_3=\aa12$.
Since $\brbr_3$ equals $\aa12$ we obtain $\brbr_4=1$ and $\br_4=\aa12$.
We conclude that the $\ff3$-splitting of $\br$ is $(\aa12,\aa12,1,\aa12^2)$.
\end{exam}

Before giving the definition of the rotating normal we fix some definitions about words.

\begin{defi}
 A word on the alphabet $\SPBKL\nn$ is an \emph{$\SPBKL\nn$-word}.
 A word on the alphabet $\SBKL\nn=\SPBKL\nn\sqcup\SPBKL\nn\inv$ is an \emph{$\SBKL\nn$-word}.
The braid represented by the $\SBKL\nn$-word $\ww$ is denoted by $\overline{\ww}$.
For $\ww$, $\ww'$ two $\SBKL\nn$-word, we say that $\ww$ is equivalent to $\ww'$, 
denoted by $\ww\equiv\ww'$ if $\overline\ww=\overline{\ww'}$ holds. 
The empty word is denoted by $\varepsilon$.
\end{defi}

The \emph{$\nn$-rotating normal form} is an injective map $\rnf_\nn$ from $\BKL\nn$ to the set of $\SPBKL\nn$-words defined inductively using the $\ff\nn$-splitting.

\begin{defi}\label{D:Rotating}
For  $\br\in\BKL2$, we define $\rnf_2(\br)$ to be the unique word~$\aa12^k$ representing $\br$.
The rotating normal form of a braid $\br\in\BKL\nn$ with $n\geq 3$ is 
$$\rnf_\nn(\br)= \ff\nn^\brdio(\rnf_\nno(\br_\brdi))\cdot...\cdot\ff\nn(\rnf_\nno(\br_2))\cdot\rnf_\nno(\br_1),$$
 where $(\br_\brdi,\Ldots,\br_1)$ is the $\ff\nn$-splitting of $\br$. 
 A word $w$ is said to be \emph{$n$-rotating} if it is the $\nn$-rotating normal form of a braid of $\BKL\nn$.
\end{defi}

As the $\nn$-rotating normal form of a braid of $\BKL\nno$ is equal to its $(\nno)$-rotating normal form we can talk without ambiguities of the \emph{rotating normal form} about a dual braid.

\begin{exam}
We reconsider the braid $\br$ of Example~\ref{E:Split}.
We know that the $\ff3$-splitting of $\br$ is $(\aa12,\aa12,1,\aa12^2)$.
Since $\rnf_2(1)=\varepsilon$, $\rnf_2(\aa12)=\aa12$ and $\rnf_2(\aa12^2)=\aa12^2$ we obtain
$$\rnf_3(\br)=\ff3^3(\aa12)\cdot\ff3^2(\aa12)\cdot \ff3(\varepsilon)\cdot \aa12^2=\aa12\aa13\aa12\aa12.$$
\end{exam}

Some properties of the rotating normal form have been established in \cite{Fromentin2008a}.
Connections, established in  \cite{Fromentin2008} and \cite{Fromentin2008a}, between the rotating normal form and the braid's ordering introduced by P.~Dehornoy are based on these properties.

We finish this section with some already known or immediate properties about $\ff\nn$-splittings and $\nn$-rotating words.

\begin{defi}
 For every nonempty word $\ww$, the last letter of $\ww$ is denoted by~$\last\ww$.
 For each nontrivial braid $\br$ in $\BKL\nn$, we define the \emph{last letter} of $\br$,
 denoted $\last\br$, to be the last letter in the rotating normal form of $\br$.
\end{defi}

\begin{lemm}[Lemma 3.2 of \cite{Fromentin2008a}]
\label{L:LastLetter}
 Assume $\nn\geq3$ and  let $(\br_\brdi,...,\br_1)$ be a $\ff\nn$-splitting
 
 $(i)$ For $\kk\geq 2$, the letter $\last{\br}_\kk$ is of type $\aa{..}\nno$ unless $\br_\kk=1$.
  
 $(ii)$ For $\kk\geq3$ and $\kk=\brdi$, we have $\br_\kk\not=1$.
 
\end{lemm}

The fact that $\br_\brdi$ is not trivial is a direct consequence of the definition of $\ff\nn$-splitting.
As, for $\kk\geq2$, the braid $\brr=\ff\nn(\last{\br}_\kkp)\br_\kk$ is a right-divisor of~$\ff\nn^{\brdi-\kk}(\br_\brdi)\cdot...\cdot\br_\kk$, it must satisfy some properties.
In particular, if $\last{\br}_\kkp=\aa\indio\nno$ holds then the $\BKL\nno$-tail of $\ff\nn(\aa\indi\nn\br_\kk)$ 
is trivial by \eqref{E2:P:Split}.

\begin{defi}
We say that a letter $\aa\indiii\indiv$ is an \emph{$\aa\indi\nn$-barrier} if $1\leq \indiii<\indi<\indiv\leq \nno$ holds.
\end{defi}

There exist no $\aa\indi\nn$-barrier with $\nn\leq 3$ and the only $\aa\indi4$-barrier is $\aa13$, which is an $\aa24$-barrier. 
By definition, if the letter $x$ is an $\aa\indi\nn$-barrier, then in the presentation of $\BKL\nn$ there exists no relation of the form $\aa{p}{n}\,\cdot\, x=y\,\cdot\, \aa{p}{n}$ allowing one to push the letter $\aa{p}{n}$ to the right through the letter~$x$: so, in some sense, $x$ acts as a barrier.

\begin{lemm}[Lemma 3.4 of \cite{Fromentin2008a}]
\label{L:Barrier}
 Assume that $\nn\geq 3$, $\br$ is a braid of $\BKL\nno$ and the $\BKL\nno$-tail of $\ff\nn(\aa\indi\nn\br)$ 
 is trivial for $2\leq\indi\leq\nnt$.
 Then the rotating normal form of $\br$ is not the empty word and it contains an $\aa\indi\nn$-barrier.
\end{lemm}

 \begin{lemm}[Lemma 3.5 of \cite{Fromentin2008a}]
 \label{L:Barrier2}
  Let $(\br_\brdi,...,\br_1)$ be a $\ff\nn$-splitting of some braid of $\BKL\nn$ with $n\geq3$.
  Then for each $\kk\in[2,\brdio]$ such that $\last{\br}_\kkp$ is not $\aa\nnt\nno$ (if any), the rotating normal
  form of $\br_\kk$ contains an $\ff\nn(\last{\br}_\kkp)$-barrier.
 \end{lemm}

\section{Left reversing for dual braid monoid}

Left reversing process was introduced by P.~Dehornoy in \cite{Dehornoy1997a}.
It is a powerfull tool tfor the investigation of division properties in some monoids as stated by Proposition~\ref{P:Div}.

\begin{defi}
 
A monoid $M$ defined by a presentation $\left<S\, |\, R\right>^+$ is \emph{left complemented} if there exists 
a map $f:S\times S\to S^\ast$ satisfying
\[
 R=\left\{f(x,y)x=f(y,x)y\, | \, (x,y)\in S^2\right\}
\]
and if $f(\xx,\xx)=\varepsilon$ holds for all $\xx\in S$.
\end{defi}

As the relation $\xx=\xx$ is always true for $\xx\in S$ we say that $M$ is left complemented even if
$x=x$ does not occur in $R$ for $x\in S$.

The monoid $\BKL3$ with presentation of Proposition~\ref{P:PresBKL} is left complemented with respect to the map $f$ given by 
\begin{align*}
 f(\aa12,\aa23)=f(\aa12,\aa13)&=\aa13\\
 f(\aa23,\aa12)=f(\aa23,\aa13)&=\aa12\\
 f(\aa13,\aa12)=f(\aa13,\aa23)&=\aa23
\end{align*}
However the monoid $\BKL4$ with presentation of Proposition~\ref{P:PresBKL} is not left complemented.
Indeed there is no relation of the form $...\,\aa13=...\,\aa24$.
Hence the words $f(\aa13,\aa24)$ and $f(\aa24,\aa13)$ are not well defined.

In general for  $1\leq\indi<\indiii<\indii<\indiv\leq\nn$, the word 
$f(\aa\indi\indii,\aa\indiii\indiv)$ and $f(\aa\indiii\indiv,\aa\indi\indii)$ are not defined for
the presentation of $\BKL\nn$ given in Proposition~\ref{P:PresBKL}.
In order to obtain a left complemented presentation of $\BKL\nn$ we must exhibit some extra relations from these
given in Proposition~\ref{P:PresBKL}.

By example, the relation $\aa23\aa14\,\aa13\equiv\aa34\aa12\,\aa24$ holds and so we can consider 
$f(\aa13,\aa24)$ to be $\aa23\aa14$. 
However the relation $\aa14\aa23\,\aa13\equiv\aa34\aa12\,\aa24$ is also satisfied and so $f(\aa13,\aa24)=\aa23\aa14$
is an other valid choice.

\begin{lemm}
\label{L:Comp}
For $n\geq2$, the map $f_n:\SPBKL\nn\times\SPBKL\nn\to\SPBKL\nn^\ast$ defined by
\[
 f_n(\aa\indi\indii,\aa\indiii\indiv)=\begin{cases}
                                       \varepsilon&\text{for $\aa\indi\indii=\aa\indiii\indiv$,}\\
                                       \aa\indi\indiv&\text{for $\indii=\indiii$,}\\
                                       \aa\indiv\indii&\text{for $\indi=\indiii$ and $\indii>\indiv$,}\\
                                       \aa\indiii\indi&\text{for $\indii=\indiv$ and $\indi>\indiii$,}\\
                                       \aa\indiii\indii\aa\indi\indiv&\text{for $\indi<\indiii<\indii<\indiv$,}\\
                                       \aa\indiv\indii\aa\indiii\indi&\text{for $\indiii<\indi<\indiv<\indii$,}\\
                                       \aa\indiii\indiv&\text{otherwise.}
                                      \end{cases}
\]
provides a structure of left complemented monoid to $\BKL\nn$.
\end{lemm}

\begin{proof}
Direct computations using  Proposition~\ref{P:PresBKL} establish 
$f_\nn(\xx,\yy)\cdot\xx\equiv f_\nn(\yy,\xx)\cdot \yy$ for all $(x,y)\in \SPBKL\nn^2$.
\end{proof}

Our choice for $f_n(\aa\indi\indii,\aa\indiii\indiv)$ with $\indi<\indiii<\indii<\indiv$ is well suited for the sequel 
and some proof would be invalid if we made an other one.

\begin{defi}
 For $\ww$ and $\ww'$ two $\SBKL\nn$-words, we say that $\ww$ \emph{left reverses in one step} to~$\ww'$, denoted $\ww\rev^{1}\ww'$,
 if we can obtain $\ww'$ from $\ww$ substituting a factor $xy\inv$ (with $x,y \in \SPBKL\nn$) by $f_n(x,y)\inv f_n(y,x)$.
 We say that $\ww$ \emph{left reverses} to $\ww'$, denoted by~$\ww\rev\ww'$, if there exists a sequence 
 $\ww=\ww_1,...,\www_\ell=\ww'$ of $\SBKL\nn$-words such that $\ww_\kk\rev^{1}\ww_{\kkp}$ for $\kk\in[1,\ell-1]$.
\end{defi}

\begin{exam}
\label{E:Rev}
 The word $\uu=\aa12\aa23\aa12\aa13\inv$ left reverses to $\aa23\aa23$ as the following left reversing sequence shows (left reversed factor are underlined)
 $$
 \aa12\aa23\underline{\aa12\aa13\inv}\rev^{1}\aa12\underline{\aa23\aa13\inv}\aa23\rev^{1}\underline{\aa12\aa12\inv}\aa23\aa23\rev^1\aa23\aa23,
 $$
which is denoted by $\aa12\aa23\aa12\aa13\inv\rev{}\aa23\aa12$.
\end{exam}

\begin{defi}
  For $\ww$ an $\SBKL\nn$-word, we denote by $D(\ww)$ and $N(\ww)$ the unique $\SPBKL\nn$-word,  if there exist, such that $\ww\rev D(\ww)\inv\,N(\ww)$. 
  The word $N(\ww)$ is the \emph{left numerator} of $\ww$ while the word $D(\ww)$ is its \emph{left denominator}.
 \end{defi}

 Reconsidering Example~\ref{E:Rev}, we obtain that the left denominator of $\uu$ is $D(\uu)=\varepsilon$ and that is left  numerator its $N(\uu)=\aa23\aa23$.

A consequence of Example~8 and Proposition~3.5 of \cite{Dehornoy1999} based on \cite{Dehornoy1997a} and \cite{Birman1998} is 
that $N(\ww)$ and $D(\ww)$ exists for any $\SBKL\nn$-word $\ww$.
We obtain also the following result: 

\begin{prop}
\label{P:Div}
 Let $\ww$ be an $\SPBKL\nn$-word and $\aa\indi\indii$ be in $\SPBKL\nn$. 
 The braid $\overline{\ww}$ is right divisible by $\aa\indi\indii$ if and only if $D(\ww\cdot\aa\indi\indii\inv)$ is empty.
\end{prop}

Since the denominator of $\aa12\aa23\aa12\aa13\inv$ is empty, the braid $\aa13$ right divides the braid $\aa12\aa23\aa12$.

\section{Characterization of rotating normal words}

The aim of this section is to give a syntactical characterization of $\nn$-rotating words among $\SPBKL\nn$-words.

\begin{defi}
We say that a braid $\br$ in $\BKL\nn$ contains an $\aa\indi\nn$-barrier if its rotating normal form does.
\end{defi}
 
 \begin{lemm}
 \label{L:UniqueLastLetter}
 Assume that $\nn\geq3$, $\br$ belongs to $\BKL\nno$ and that the $\BKL\nno$-tail of $\ff\nn(\br)$ is trivial.
 Then every $\SPBKL\nno$-word representing $\br$ ends with $\last\br$.
\end{lemm}

\begin{proof}
 Let $\uu$ be an $\SPBKL\nno$-word representing $\br$. 
 As the $\BKL\nno$-tail of $\ff\nn(\br)$ is trivial, the last letter $\last\uu$ of $\uu$ not belongs to $\SPBKL\nnt$ 
 and so $\last\uu$ is $\aa\indi\nno$ for some integer $\indi<\nno$.
 Assume now $\vv$ is an other $\SPBKL\nno$-word representing~$\br$. 
 For the  same reason as $\uu$, we have $\last\vv=\aa\indii\nno$ for some $\indii<\nno$.
 Since the two braids $\aa\indi\nno$ and $\aa\indii\nno$ are right divisors of $\br$, their left lcm is also
 a right divisor of $\br$.
 Assume for a  contradiction that $\indi$ and $\indii$ are different.
 The braid $\br$ is then right divisible by $\aa\indi\indii\aa\indii\nno$, which is 
 the left lcm of $\aa\indi\nno$ and $\aa\indii\nno$. 
 Since $\aa\indi\indii\aa\indii\nno$ is equivalent to $\aa\indi\nno\aa\indi\indii$, the braid $\aa\indi\indii$ is also a 
 right divisor of $\br$.
 In particular $\aa\indip\indiip$, with $\indiip<n$, is  a right divisor of $\ff\nn(\br)$,
 which is impossible since the $\BKL\nno$-tail of~$\ff\nn(\br)$ is supposed to be trivial.
 Therefore, every $\SPBKL\nno$-word representing $\br$ ends with the same letter, namely~
$\last\br$.
\end{proof}

We conclude that, under some hypotheses, the last letter of a word is a braid invariant.

 \begin{defi}
 \label{D:Ladder}
For $\nn\geq 3$ and $2\leq \indi \leq \nno$, we say that an $\nn$-rotating word $\ww$ is an \emph{$\aa\indi\nn$-ladder} is there exist a decomposition 
$$
w=v_0\,x_1\,v_1\,\ldots\,v_{h-1}\,x_h\,v_h,
$$
a sequence $p=j(0)<j(1)<...<j(h)=\nno$ and a sequence $i$ such that

$(i)$ for each $k\leq h$, the letter $x_k$ is $\aa{i(k)}{j(k)}$ with $i(\kk)<j(\kko)<j(\kk)$,

$(ii)$ for each $k<h$, the word $v_k$ contains no $\aa{j(k)}\nn$-barrier,
\end{defi}

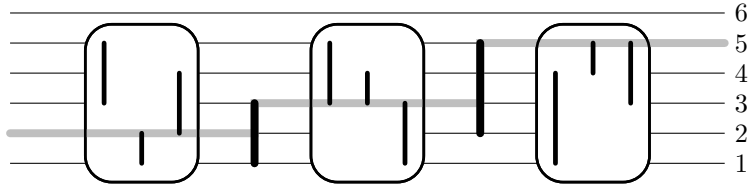
\begin{figure}[h!]
\label{F:Ladder}
  \begin{tikzpicture}[x=0.05cm,y=0.05cm]
    \foreach \v in {1,2,3,4,5,6} \draw(0,\v*8-8) -- (190,\v*8-8) node[right]{\small $\v$};
    \foreach \x in {0,1,2} \filldraw[rounded corners=10,fill=white,line width=1](60*\x+20,-5) rectangle ++ (30,42);
    \draw[line width=3,color=lightgray,line cap=round](0,8) -- (65,8) -- (65,16) -- (125,16) -- (125,32) -- (190,32);
  \foreach \x in {0,1,2} \draw[rounded corners=10,line width=1](60*\x+20,-5) rectangle ++ (30,42);
    \draw[line width=3,line cap=round](65,0) -- (65,16);
    \draw[line width=3,line cap=round](125,8) -- (125,32);
    \draw[line width=1.8,line cap=round] (25,16) -- (25,32);
    \draw[line width=1.8,line cap=round] (35,0) -- (35,8);
    \draw[line width=1.8,line cap=round] (45,8) -- (45,24);
    \draw[line width=1.8,line cap=round] (85,16) -- (85,32);
    \draw[line width=1.8,line cap=round] (105,0) -- (105,16);
    \draw[line width=1.8,line cap=round] (95,16) -- (95,24);
    \draw[line width=1.8,line cap=round] (145,0) -- (145,24);
    \draw[line width=1.8,line cap=round] (155,24) -- (155,32);
    \draw[line width=1.8,line cap=round] (165,16) -- (165,32);
  \end{tikzpicture}
  \caption{An $\aa26$-ladder. The gray line starts at position~$1$ and goes up to position $5$ using the bar of the ladder. The empty spaces between bars in the ladder are represented by a framed box. In such boxes the vertical line representing the letter $\aa\ii\jj$ does not cross the gray line. The bar of the ladder are represented by black thick vertical lines.}
\end{figure}

Condition $(ii)$ is equivalent to: for each $\kk\leq \hh$, the letter $\xx_\kk$ is an $\aa{j(\kko)}\nn$-barrier of type $\aa{..}{j(\kk)}$.

An immediate adaptation of Proposition~3.9 of \cite{Fromentin2008a} is :

\begin{lemm}
 \label{L:Barrier:Ladder}
 Assume that $\nn\geq 3$, $\br$ belongs to $\BKL\nno$, the $\BKL\nno$-tail of $\ff\nn(\br)$ is trivial and $\br$ contains an  $\aa\indi\nn$-barrier for some $2\leq \indi\leq \nnt$. Then the normal form of $\br$ is an $\aa\indi\nn$-ladder.
\end{lemm}

In order to obtain a syntactical characterization of $\nn$-rotating words we want a local version of condition~\eqref{E2:P:Split} characterizing a $\ff\nn$-splitting.
The following result is the first one in this way.

\begin{prop}
\label{P:Equiv}
  For $\br\in\BKL\nno$ and $\indi$ an integer satisfying $2\leq\indi\leq\nnt$ there is equivalence between

  $(i)$ the $\BKL\nno$-tail of $\ff\nn(\aa\indi\nn\br)$ is trivial,

  $(ii)$ the $\BKL\nno$-tail of $\ff\nn(\br)$ is trivial and $\br$ contains an $\aa\indi\nn$-barrier,

  $(iii)$ the only $\SPBKL\nn$-letter that right divides $\aa\indi\nn\br$ is $\last{\br}$, which is of type~$\aa{..}\nno$.
\end{prop}

Our proof of Proposition~\ref{P:Equiv} rests on the following Lemma.

\begin{lemm}
 \label{L:Reversing}
For $n\geq3$, $\uu$ an $\SPBKL\nno$-word and $\indi\in[1,\nno]$, the left denominator $D(\uu\aa\indi\nn\inv)$ is not empty. More precisely, $D(\uu\aa\indi\nn\inv)\inv$ begins with $\aa\indii\nn\inv$ satisfying $\indii\leq \indi$.
\end{lemm}

\begin{proof}
 Assume that $\ww_1,...,\ww_\ell$ is a reversing sequence from the word $\ww_1=\uu\aa\indi\nn\inv$ to the word $D(\ww_1)\inv N(\ww_1)$. For $\kk\in[1,\ell]$ we denote by $\yy_\kk$ the leftmost negative letter in $\ww_\kk$. Each reversing step consists in replacing a factor $\xx\yy\inv$ of $\ww_\kk$ by $f_\nn(\xx,\yy)\inv f_\nn(\yy,\xx)$. If for $\kk\in[1,\ell]$ the reversed factor of $\ww_\kk$ does not contains $\yy_\kk\inv$ then $\yy_\kkp$ equals $\yy_\kk$. Assume now that the reversed factor is $\xx\yy_\kk\inv$ with $\yy_\kk=\aa\indiii\nn$. Lemma~\ref{L:Comp} implies
 \[
 f_n(x,\yy_\kk)=f_n(\aa\ii\jj,\aa\indiii\nn)=
 \begin{cases}
  \aa\ii\nn&\text{for $\jj=\indiii$,}\\
  \aa\indiii\jj\aa\ii\nn&\text{for $\ii<\indiii<\jj$,}\\
  \aa\indiii\nn&\text{otherwise,}
 \end{cases}
 \]
 which gives in particular
 \begin{equation}
 \label{E:Reversing}
 \xx\yy_\kk\inv=\aa\ii\jj\aa\indiii\nn\inv\rev
  \begin{cases}
   \aa\ii\nn\inv...&\text{for $\ii<\indiii\leq\jj$,}\\
   \aa\indiii\nn\inv...&\text{otherwise.}
  \end{cases}
 \end{equation}                            
 It follows that $\yy_\kkp$ is equal to $\aa\indiv\nn$ for some $\indiv\leq\indiii$. Eventually we obtain that  $\uu\aa\indi\nn\inv$ left reverses to $\aa\indii\nn\inv...$ with the relation $\indii\leq\indi$ and so the desired property on~$D(\uu\aa\indi\nn)$ holds.
\end{proof}

\begin{proof}[Proof of Proposition~\ref{P:Equiv}]
Assume $(i)$. 
As the the $\BKL\nno$-tail of $\ff\nn(\br)$ is also a right divisor of $\ff\nn(\aa\indi\nn\,\br)$ the first statement of $(ii)$ holds. The second statement is Lemma~\ref{L:Barrier}. 
Let us prove that $(iii)$ implies $(i)$. By hypothesis  the last letter of $\br$ is $\aa\indiio\nno$ for some $\indii$. As the only $\SPBKL\nn$-letter that right divides~$\ff\nn(\aa\indi\nn\br)$ is $\ff\nn(\aa\indiio\nno)=\aa\indii\nn$, the $\BKL\nno$-tail of $\ff\nn(\aa\indi\nn\,\br)$ must be trivial.
 
 We now prove  $(ii)\Rightarrow(iii)$. Since the $\BKL\nno$-tail of~$\ff\nn(\br)$ is trivial, the letter~$\last\br$ must be of type $\aa{..}\nno$. We denote by $\ww$ the rotating normal from of $\br$.
 Let $\aa\indiii\indiv$ be an $\SPBKL\nn$-letter different from $\last\br$. 
 We will show that $\aa\indiii\indiv$ cannot be a right divisor of $\aa\indi\nn\,\br$.
 Assume first $\indiv\leq\nno$. By Lemma~\ref{L:UniqueLastLetter}, $\aa\indiii\indiv$ is not a right divisor of $\br$. Proposition~\ref{P:Div} implies that the word~$D(\ww\,\aa\indiii\indiv\inv)$ must be non empty. As the reversing of an $\SBKL\nno$-word is also an $\SBKL\nno$-word, there exists a letter~$\aa\indv\indvi$ with $\indvi<n$ such that 
 \[
 \aa\indi\nn\,\ww\,\aa\indiii\indiv\inv\rev\aa\indi\nn\,\aa\indv\indvi\inv...,
 \]
 holds. Clearly, the braid $\aa\indv\indvi$ is not a right divisor of $\aa\indi\nn$ (since we have $\indvi<\nn$). 
 Therefore, by Proposition~\ref{P:Div}, the left denominator of $\aa\indi\nn\,\ww\,\aa\indiii\indiv\inv$ is not empty, and we conclude that $\aa\indiii\indiv$ is not a right divisor of $\aa\indi\nn\br$.

 Assume now $\indiv=\nn$. Hypotheses on $\br$ plus Lemma~\ref{L:Barrier:Ladder} imply that $\ww$ is an $\aa\indi\nn$-ladder. Following Definition~\ref{D:Ladder}, we write 
 $$\ww=\vv_0\,\xx_1\,\vv_1\,...\,\vv_{\hho}\,\xx_\hh\,\vv_{\hh}.$$	
 By Lemma~\ref{L:Reversing}, there exist two  maps $\eta$ and $\mu$ from $\NN$ to itself such that 
 $$
 \ww\aa\indiii\nn\inv=\ww_\hh\aa{\eta(\hh)}\nn\inv\rev\www_\hh\aa{\mu(\hh)}\nn\inv\cdots\rev...\rev\ww_0\aa{\eta(0)}\nn\inv\cdots\rev\www_0\aa{\mu(0)}\nn\inv,
 $$
 where for all $\kk\in[0,\hh]$,
 \begin{align*}
 \ww_\kk&=\vv_0\,\xx_1\,\vv_1\,...\,\vv_{\kko}\,\xx_\kk\,\vv_\kk,\\
 \www_\kk&=\vv_0\,\xx_1\,\vv_1\,...\,\vv_{\kko}\,\xx_\kk.
 \end{align*}
 By construction $\ww_0$ is $\vv_0$ while $\www_0$ is the empty word.
 Lemma~\ref{L:Reversing} implies
 \begin{equation}
 \label{E:Inter1}
 \mu(0)\leq\eta(0)\leq\mu(1)\leq...\leq\mu(\hh)\leq\eta(\hh)=\indiii.
 \end{equation}
  Following Definition~\ref{D:Ladder} we write $\xx_\kk=\aa{\ii(\kk)}{\jj(\kk)}$.
 We will now prove by induction
 \begin{equation}
 \label{E:Inter2}
 \text{for all $\kk\in[0,\hho]$, $\mu(\kkp)\leq j(\kkp)\Rightarrow \eta(\kk)<j(\kk)$}
 \end{equation}
 Let $\kk\in[0,\hho]$ and assume $\mu(\kkp)\leq j(\kkp)$. Definition~\ref{D:Ladder} $(i)$ guarantees the relation $\ii(\kkp)<\jj(\kk)<\jj(\kkp)$.
 For $\mu(\kkp)\leq \ii(\kkp)$ we have 
 $$\eta(\kk)\leq\mu(\kkp)\leq \ii(\kkp)<\jj(\kk),$$
 and we are done in this case.
 The remaining case is $\mu(\kkp)>\ii(\kkp)$. 
 By relation~\eqref{E:Reversing}, with $\ii=\ii(\kkp)$, $\jj=\jj(\kkp)$ and $r=\eta(\kkp)$ we obtain 
 $$\xx_{\kkp}\,\aa{\mu(\kkp)}\nn\inv=\aa{\ii(\kkp)}{\jj(\kkp)}\,\aa{\mu(\kkp)}\nn\inv\rev \aa{\ii(\kkp)}\nn\inv \vv$$
 with some $\SBKL\nn$-word $\vv$.
 In particular, we have $\eta(\kk)=\ii(\kkp)<\jj(\kk)$ and~\eqref{E:Inter2} is established. 
 For $\kk=\hho$ the left hand member of \eqref{E:Inter2} is satisfied since $j(\hh)$ is equal to $\nno$ and $\indiii\leq\nno$ holds by definition of $\indiii$. 
 Properties~\eqref{E:Inter1} and \eqref{E:Inter2} imply $\mu(\kk)<\jj(\kk)$ for all $\kk\in[0,\hht]$.
 In particular we have $\mu(0)<j(0)=\indi$ together with $\ww\aa\indiii\nn\inv\rev\aa{\mu(0)}\nn\inv\cdots$. 
 As $\aa{\mu(0)}\nn$ can not be a right divisor of $\aa\indi\nn$ it follows that the left denominator of $\aa\indi\nn\,\ww\,\aa\indiii\nn\inv$ is also non empty and so that $\aa\indiii\nn$ is not a right divisor of $\aa\indi\nn\br$.
\end{proof}

As the reader can see, the case $\indi=\nno$ is excluded from Proposition~\ref{P:Equiv}.
It is the aim of the following result.

\begin{prop}
 \label{P:Equiv2}
 For $\br$ a non-trivial braid of $\BKL\nno$ there is equivalence between 
 
 $(i)$ the $\BKL\nno$-tail of $\ff\nn(\aa\nno\nn\br)$ is trivial,
 
 $(ii)$ the only $\SPBKL\nn$-letter right divising $\aa\nno\nn\br$ is $\last{\br}$ which is of type $\aa{..}\nno$.
\end{prop}	

\begin{proof}
 $(ii)\Rightarrow (i)$ is similar as $(iii)\Rightarrow (i)$ of Proposition~\ref{P:Equiv2}. We now show that $(i)$ implies~$(ii)$. 
 Condition $(i)$ implies in particular that the $\BKL\nno$-tail of $\ff\nn(\br)$ is trivial. 
 It follows that the last letter of $\br$ is of type $\aa{..}\nno$.
 Let $\ww$ be the rotating normal form of $\br$ and $\aa\indiii\indiv$ be an $\SPBKL\nn$-letter different from $\last\br$. 
 For $\indiv\leq\nno$ we follow proof of Proposition~\ref{P:Equiv2} to obtain that $\aa\indiii\indiv$ is not a right divisor or the braid $\aa\nno\nn$.
 Assume now $\indiv=\nn$. 
 By Lemma~\ref{L:Reversing} there exists $\indii\leq \indiii$ such that $\ww\aa\indiii\nn\inv\rev\aa\indii\nn\inv\cdots$ holds and so we obtain $\aa\nno\nn\ww\aa\indiii\nn\inv\rev\aa\nno\nn\aa\indii\nn\inv\cdots$. 
 As, for $\indii\not=\nno$ the braid $\aa\indii\nno$ is not a right divisor of $\aa\nno\nn$ it is sufficient to show $\indii\not=\nno$ for concluding that $\aa\indiii\nn$ not right divides $\aa\nno\nn\br$.
 For $\indiii\leq\nnt$ it is obvious since $\indii\leq\indiii$ holds.
 Assume finally $\indiii=\nno$. We denote by $\aa\indi\nno$ the last letter of $\br$. By~\eqref{E:Reversing} we have $\aa\indi\nno\aa\nno\nn\inv\rev\aa\indi\nn\inv\cdots$ and then Lemma~\ref{L:Reversing} implies $\indii\leq\indi<\nno$, as expected.
\end{proof}

\begin{thrm}
\label{T:Main}
 A finite sequence $(\br_\brdi,...,\br_1)$ of braids in $\BKL\nno$ is the $\ff\nn$-splitting of a braid of $\BKL\nn$ if and only if
 
 $(i)$ for $\kk\geq 3$ and $\kk=\brdi$, the braid $\br_\kk$ is not trivial,
 
 $(ii)$ for $\kk\geq2$, the $\BKL\nno$-tail of $\ff\nn(\br_\kk)$ is trivial,
 
 $(iii)$ if, for $\kk\geq3$, we have $\last{\br}_\kk\not=\aa\nnt\nno$ then $\br_\kko$ contains an $\ff\nn(\last{\br}_\kk)$-barrier.
\end{thrm}

\begin{proof}
 Let $(\br_\brdi,...,\br_1)$ be the $\ff\nn$-splitting of some braid of $\BKL\nno$. 
 Condition~$(i)$ is a consequence of  Lemma~\ref{L:LastLetter}.$(ii)$. Condition~\eqref{E2:P:Split}  implies that the $\BKL\nno$-tail of
 $$
 \ff\nn^{\brdi-\kk}(\br_\brdi)\cdot...\cdot\ff\nn(\br_\kkp)
 $$
 is trivial for $\kk\geq1$. In particular the $\BKL\nno$-tail of $\ff\nn(\br_\kkp)$ must be trivial for $\kk\geq 1$, which implies $(ii)$. Condition $(iii)$ is Lemma~\ref{L:Barrier2}.
 
 Conversely, let us prove that a sequence $(\br_\brdi,...,\br_1)$ of braids of $\BKL\nno$ satisfying $(i),(ii)$ and $(iii)$ is the $\ff\nn$-splitting of some braid of $\BKL\nn$. 
 Condition~$(i)$ implies that $\br_\brdi$ is not trivial. 
 For $\kk\geq 2$ we denote by $\brbr_\kk$ the braid $\ff\nn^{\brdi-\kk}(\br_\brdi)\cdot...\cdot\ff\nn(\br_\kkp)\cdot\br_\kk$.
 For $\kk\geq 3$ and $\kk\geq2$ whenever $\br_2\not=1$, we first prove
 \begin{equation}
 \label{E:ToProve}
  \text{$\last{\br}_\kk$ is the only $\SPBKL\nn$-letter that right divides $\brbr_\kk$.}
 \end{equation}
 We note that Condition $(i)$ guarantees the existence of $\last{\br}_\kk$ for $\kk\geq 3$.
  For $\kk=\brdi$, Condition $(ii)$ implies that the $\BKL\nno$-tail of $\ff\nn(\br_\brdi)$ is trivial. 
  Hence, by
  Lemma~\ref{L:UniqueLastLetter} the only $\SPBKL\nno$-letter that right divides $\br_\brdi$ is $\last{\br}_\brdi$. 
 Since any right divisors of a braid of~$\BKL\nno$ lie in $\BKL\nno$, we have established \eqref{E:ToProve} for $\kk=\brdi$.
 Assume \eqref{E:ToProve} holds for $\kk\geq4$ or $\kk\geq3$ whenever $\br_2\not=1$ and let us prove it for $\kko$. 
 By Condition $(ii)$ there exists $\indi$ such that $\last\br_\kk$ is~$\aa\indio\nno$.
 We denote by $\uu\aa\indio\nno$ and $\vv$ two $\SPBKL\nn$-words representing $\brbr_\kk$ and $\br_\kko$ respectively.
 The braid $\brbr_\kko$ is then represented by $\ff\nn(\uu)\aa\indi\nn\vv$. 
 Let $\yy$ be an $\SPBKL\nn$-letter different from $\last\br_\kko$. 
 Proposition~\ref{P:Equiv} with Condition~$(iii)$ and Proposition~\ref{P:Equiv2} imply that $\yy$ is not a right divisor of $\aa\indi\nn\,\br_\kko$. 
 Therefore, by Proposition~\ref{P:Div} there exists  an $\SPBKL\nn$-letter $\xx$ different from $\aa\indi\nn$ such that $\ff\nn(\uu)\aa\indi\nn\vv\yy\inv\rev\ff\nn(\uu)\aa\indi\nn\xx\inv\cdots$. 
 The word $\ff\nn(\uu)\aa\indi\nn$ represents $\ff\nn(\brbr_\kk)$. 
 By induction hypothesis $\xx$ is not a right divisor of $\ff\nn(\brbr_\kk)$. 
 Then Proposition~\ref{P:Div} implies that $D(\ff\nn(\uu)\aa\indi\nn\xx\inv)$ is not empty. 
 It follows $D(\ff\nn(\uu)\aa\indi\nn\vv\yy\inv)\not=\varepsilon$ and so always by Proposition~\ref{P:Div}, the letter $\yy$ is not a right divisor of $\brbr_\kko$. 
 Eventually we have established~\eqref{E:ToProve} for $\kk\geq3$.

A direct consequence of \eqref{E:ToProve} and Condition $(ii)$ is that the only $\SPBKL\nn$-letter right divising $\ff\nn(\brbr_\kk)$ is of type $\aa{..}\nn$ and so the $\BKL\nno$-tail of he braid~$\brbr_\kk$ is trivial for $\kk\geq 3$ and for $\kk=2$ whenever $\br_\kk\not=1$. 
It remains to establish that the $\BKL\nno$ tail of $\ff\nn(\brbr_2)$ is also trivial whenever $\br_2$ is trivial.
Assume $\br_2=1$. 
Condition $(iii)$ implies $\last\br_3=\aa\nnt\nno$.
By \eqref{E:ToProve}, $\aa\nnt\nno$ is the only $\SPBKL\nn$-letter that right divides the braid~$\brbr_3$. 
Since $\brbr_2=\ff\nn(\brbr_3)$, the letter~$\ff\nn^2(\aa\nnt\nno)=\aa1\nn$ is the only letter right divising $\ff\nn(\brbr_2)$. In particular the $\BKL\nno$-tail of~$\ff\nn(\brbr_2)$ is trivial.
\end{proof}

Conditions $(i)$, $(ii)$ and $(iii)$ are easy to check if the braids $\br_1,...,\br_\brdi$ are given by their rotating normal forms.

\begin{coro}
\label{C:Main}
Let $(\ww_\brdi,...,\ww_1)$ be a finite sequence of $\SPBKL\nno$-words, then the word 
\begin{equation}
\label{E:C:Main}
\ff\nn^\brdio(\ww_\brdi)\cdot...\cdot\ff\nn(\ww_2)\cdot\ww_1,
\end{equation}
is $\nn$-rotating if the following conditions are satisfied

$(i)$ for $\kk\geq1$, the word $\ww_\kk$ is $(\nno)$-rotating,

$(ii)$ for $\kk\geq3$, the word $\ww_\kk$ ends by $\aa\indio\nno$ for some $\indi$,

$(iii)$ the word $\ww_2$ is either empty (except for $\brdi=2)$ or ends by $\aa\indio\nno$ for some~$\indi$,

$(iv)$ if, for $\kk\geq3$, the word $\ww_\kk$ ends by $\aa\indio\nno$ with $\indi\not=\nno$ then the word~$\ww_\kko$ contains an $\aa\indi\nn$-barrier.
\end{coro}

\begin{proof}
Assume that $(\ww_\brdi,...,\ww_1)$ satisfies Conditions $(i)$-$(iv)$ and let us prove that the word $\ww$ defined at \eqref{E:C:Main} is rotating.

We denote by $\br_\ii$ (resp. $\br$) the braid represented by $\ww_\ii$ (resp. $\ww$). By Condition $(i)$ and  Definition~\ref{D:Rotating}, the word $\ww$ is rotating if and only if $(\br_\brdi,...,\br_1)$ is a $\ff\nn$-splitting. Conditions $(ii)$ and $(iii)$ imply Condition $(i)$ of Theorem~\ref{T:Main}. Theorem~\ref{T:Main}.$(iii)$ is a consequence of $(ii)$ and $(iv)$. We remark that the $\BKL\nno$-tail of a braid $\brbr$ is represented by a suffix of the rotating word of $\brbr$. 
If the $\BKL\nno$-tail of $\ff\nn(\br_\kk)$ is not trivial, then there exists $\aa\indi\indii$ with $1\leq\indi<\indii<\nn$ that right divides $\ff\nn(\br_\kk)$. 
As $\br_\kk$ lies in $\BKL\nno$, we have $\indi\not=1$ and therefore $\br_\kk$ is right divisible by $\aa\indio\indiio$ with $\indiio\leq\nnt$. Assume the $\BKL\nno$-tail of $\ww_\kk$ is not trivial for $\kk\geq 2$. The previous remark implies that $\ww_\kk$ must end with a letter $\aa\ii\jj$ satisfying $\jj\leq\nnt$, which is in contradiction with Conditions $(ii)$ and $(iii)$.
\end{proof}

It is not true that any decomposition of an $\nn$-rotating word as in~\eqref{E:C:Main} satisfies Conditions $(i)-(iv)$ of Corollary~\ref{E:C:Main}. However we have the following result.

\begin{prop}
\label{P:Main}
 For every $\nn$-rotating word $\ww$ with $\nn\geq 3$ there exists a unique sequence $(\ww_\brdi,...,\ww_1)$ of $(\nno)$-rotating words such that $\ww$ decompose as in \eqref{E:C:Main} and Conditions $(ii)-(iv)$ of \ref{C:Main} hold.
\end{prop}

\begin{proof}
By definition of a rotating normal word and by Lemma~\ref{L:LastLetter} such a sequence exists. 
Let us prove now the unicity. 
Assume $\ww$ is a $\nn$-rotating normal word and that $(\ww_\brdi,...,\ww_1)$ and $(\www_\brdii,...,\www_1)$ are two different sequences of ($\nno$)-rotating normal words satisfying Conditions $(ii)$ and $(iii)$ of Corollary~\ref{C:Main}. 
Let $\kk$ be the minimal integer satisfying $\ww_\kk\not=\ww'_\kk$. 
Since the sum of the word lengths of the two sequences are the same, we have $\kk\leq\min\{\brdi,\brdii\}$. 
Without lost of generality, we may assume that $\www_\kk$ is a proper suffix of $\ww_\kk$, \ie,  $\ww_\kk=\uu\cdot\www_\kk$.
By Conditions $(ii)$ and $(iii)$ of Corollary~\ref{C:Main}, the last letter~$\xx$ of $\uu$ comes from the last letter of $\www_\kkp$ or $\www_\kkpp$. Hence the letter $\xx$ is equal to $\aa\indio\nno$ for some $\indi$ and $\ww_\kk$ admits either $\ff\nn(\aa\indio\nno)\www_\kk=\aa\indi\nn\www_\kk$ or $\ff\nn^2(\aa\indio\nno)\www_\kk=\aa1\indip\www_\kk$ as suffix. The first case is impossible since $\ww_\kk$ is an $\SPBKL\nno$-word. The second case may occur only for $\kk=1$ and $\www_2=\varepsilon$.
As $\www_2$ is empty, the last letter of $\www_3$, which is $\xx$, is equal to $\aa\nnt\nno$.
This implies that $\ww_\kk$ admits $\aa1\nn\uu$ as suffix which is also impossible since it is an $\SPBKL\nno$-word.
\end{proof}

A direct consequence of Corollary~\ref{C:Main} and Proposition~\ref{P:Main} is

\begin{thrm}
\label{T:Main2}
 An $\SPBKL\nn$-word $\ww$ is rotating if and only if it can be expressed as in~\eqref{E:C:Main} subject to Conditions $(i)-(iv)$ of Corollary~\ref{C:Main}.
\end{thrm}

\section{Regularity}

In this section we will show that the language of $\nn$-rotating words, denoted by~$R_\nn$ is regular, \ie, there exists a finite state automaton recognizing the $\nn$-rotating words. 
As the rotating normal form is defined using right division it is more natural for an automaton to read word from the right. For $\ww=\xx_0\cdot...\cdot\xx_\kk$ an $\SPBKL\nn$-word we will denote by $\Pi(\ww)$ the word $\xx_\kk\cdot...\cdot\xx_0$. By Theorem~1.2.8 of~\cite{Epstein1992} the language $R_\nn$ is regular if and only if the language~$\Pi(R_\nn)$ is. In this section we will construct an automaton recognizing $\Pi(R_\nn)$.

For us a \emph{finite state automaton} is a quintuplet $(S\cup\{\otimes\},A,\mu,Y,i)$ where~$S$ is the finite set of \emph{states}, $A$ is a finite \emph{alphabet}, $\mu:S\times A\to S$ is the \emph{transition function}, $Y\subseteq S$ is \emph{acceptating states} and $i$ is the \emph{initial state}. In this paper each automaton is equipped with an undraw dead state $\otimes$ and all states except the dead one is accepting, \ie, $Y=S$ always holds. Therefore an automaton will be briefly denoted by $\mathcal{A}=(S,A,\mu,i)$. To describe $\mathcal{A}$ it is enough to describe $\mu$ on $(s,x)\in S\times A$ where $\mu(s,x)\not=\otimes$ and $s\not=\otimes$. By example an automaton recognizing the language $R_2$ is $\mathcal{A}_2=(\{1\},\{\aa12\},\mu_2,1)$ with $\mu_2(1,\aa12)=1$.
The corresponding automaton diagram is :

\begin{center}
\begin{tikzpicture}[x=0.05cm,y=0.05cm]
\sstate{A}{(0,0)}{1}
\myloop{A}{\aaa12}{90}{-6}
\initstate{A}{180}
\end{tikzpicture}
\end{center}
The horizontal arrow points to the initial state.

\begin{prop}
 An $\SPBKL3$-word $x_b^{e_b}\cdot...\cdot\aa13^{e_3}\,\aa23^{e_2}\,\aa12^{e_1}$ where 
 $$
 x_b=\begin{cases}
      \aa12 & \text{if $b\equiv 1\mod 3$,}\\
      \aa23 & \text{if $b\equiv 2\mod 3$,}\\
      \aa13 & \text{if $b\equiv 3\mod 3$.}
     \end{cases}
 $$
 is rotating  if and only if $e_\kk\not=0$ for all $\kk\geq3$.
\end{prop}

\begin{proof}
 The $2$-rotating words are powers of $\aa12$. Let $\ww$ be the word of the statement. Defining $\ww_\kk$ to be $\aa12^{e_k}$, we obtain
 $$
 \ww=\ff\nn^\brdio(\ww_\brdi)\cdot...\cdot\ff\nn(\ww_2)\cdot\ww_1.
 $$
 As there is no barrier in $\BKL3$, the word $\ww$ is rotating if and only if it satisfies Conditions $(ii)$ and $(iii)$ of Corollary~\ref{C:Main}, \ie, the exponent $e_\kk$ is not $0$ for $\kk\geq 3$.
\end{proof}

As a consequence the following automaton recognizes the language $\Pi(R_3)$:

\begin{center}
\begin{tikzpicture}[x=0.05cm,y=0.05cm,every path/.style={->,>=latex,line width=0.8,label distance=0cm}]
\sstate{A1}{(240:20)}{1};
\sstate{A2}{(120:20)}{2};
\sstate{A0}{(0:20)}{3};
\sstate{I}{(-40,0)}{0};
\initstate{I}{220};
\myloop{I}{\aaa12}{120}{-9};
\myloop{A1}{\aaa23}{240}{-8};
\myloop{A2}{\aaa13}{120}{-9};
\myloop{A0}{\aaa12}{0}{-6};
\path (I) edge node[above left=-3pt] {$\scriptstyle\aaa13$} (A2);
\path (I) edge node[below left=-3pt] {$\scriptstyle\aaa23$} (A1);
\path (A1) edge [bend left] node[left=-2pt] {$\scriptstyle\aaa13$} (A2);
\path (A2) edge [bend left] node[above right=-3pt] {$\scriptstyle\aaa12$}(A0);
\path (A0) edge [bend left] node[below right=-3pt] {$\scriptstyle\aaa23$}(A1);
\end{tikzpicture}
\end{center}

Unfortunately, for $\nn\geq4$ there is no so simple characterization of $\nn$-rotating words. We will describe an inductive construction for an automaton recognizing language $\Pi\left(R_n\right)$. The process will be illustrated on $\nn=4$. The first step is to focus on $\nn$-rotating words ending with a letter of type $\aa{..}\nn$.

\begin{defi}
 We denote by $R_\nn^\ast$ the language of $\nn$-rotating words which are empty or ends with a letter of type $\aa\indi\nn$ for some $\indi$.
\end{defi}

Before constructing an automaton $\mathcal{A}_\nn$ recognizing the language $\Pi\left(R_\nn\right)$, we construct by induction on $n\geq 3$  an automaton $\mathcal{A}_\nn^\ast$ for the language $\Pi\left(R_\nn^\ast\right)$.

\begin{defi}
 A \emph{partial automaton} is a quadruplet $P=(S,A,\mu,I)$ where~$S$, $A$ and $\mu$ are defined as for an automaton and $I:A\to S$ is a map. The closure of a partial automaton $P$ is the automaton $\mathcal{A}(P)=(S\cup\{\circ\},A,\mu^c,\circ)$ given by
 $$
 \mu^c(s,x)= \begin{cases}
 I(x) & \text{if $s=\circ$,}\\
\mu(s,x) & \text{otherwise.}
 \end{cases}
$$
\end{defi}

A partial automaton is represented as an automaton excepted for the function $I$. 
For each $x\in A$ we draw an arrow attached to state $I(x)$ and labelled~$x$.
We say that a partial automaton recognizes a given language if its closure does.

\begin{figure}[h!]

\begin{center}
\begin{tikzpicture}[x=0.05cm,y=0.05cm,every path/.style={->,>=latex,line width=0.8,label distance=0cm}]
 \sstate{A1}{(240:20)}{1};
\sstate{A2}{(120:20)}{2};
\sstate{A0}{(0:20)}{3};
\myloop{A1}{\aaa23}{240}{-8};
\myloop{A2}{\aaa13}{120}{-9};
\myloop{A0}{\aaa12}{0}{-6};
\initstatel{A1}{180}{\aaa23}
\initstatel{A2}{180}{\aaa13}
\path (A1) edge [bend left] node[left=-2pt] {$\scriptstyle\aaa13$} (A2);
\path (A2) edge [bend left] node[above right=-3pt] {$\scriptstyle\aaa12$}(A0);
\path (A0) edge [bend left] node[below right=-3pt] {$\scriptstyle\aaa23$}(A1);
\begin{scope}[shift={(120,0)}]

\sstate{A1}{(240:20)}{1};
\sstate{A2}{(120:20)}{2};
\sstate{A0}{(0:20)}{3};
\sstate{I}{(-40,0)}{\circ};
\initstate{I}{180};
\myloop{A1}{\aaa23}{240}{-8};
\myloop{A2}{\aaa13}{120}{-9};
\myloop{A0}{\aaa12}{0}{-6};
\path (I) edge node[above left=-3pt] {$\scriptstyle\aaa13$} (A2);
\path (I) edge node[below left=-3pt] {$\scriptstyle\aaa23$} (A1);
\path (A1) edge [bend left] node[left=-2pt] {$\scriptstyle\aaa13$} (A2);
\path (A2) edge [bend left] node[above right=-3pt] {$\scriptstyle\aaa12$}(A0);
\path (A0) edge [bend left] node[below right=-3pt] {$\scriptstyle\aaa23$}(A1);
\end{scope}
\end{tikzpicture}
\end{center}
\caption{The partial automaton $P_3$ and the corresponding closure  which recognizes the language $\Pi(R_3^\ast)$.}
\label{F:P3}
\end{figure}
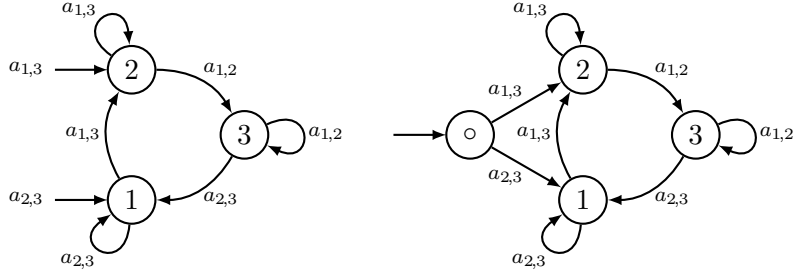

We will now show how to construct by induction a partial automaton $P_\nn$ recognizing $\Pi\left(R_\nn^\ast\right)$ for $\nn\geq 3$. For $\nn=3$ this is already done by Figure~\ref{F:P3}. For the sequel we assume $n\geq 4$ and that $P_\nno=(S_\nno,\SPBKL\nno,\mu_\nno,I_\nno)$ is a given partial automaton which recognizes the language $\Pi\left(R_\nno\right)$.

We define $S_\nn^0$ to be the set 
$$S_n^0=\{0\}\times (S_\nno\setminus\{\otimes\})\times\mathcal{P}(\{\aa2\nn,...,\aa\nnt\nn\}).$$ 
A state in~$S_n^0$ is then written $(0,s,m)$. For $x=\aa\ii\jj\in\SPBKL\nno$ we denote by $\bar\xx$ the set $\{\aa\indi\nn\,|\,\ii<\indi<\jj\}$.

 \begin{defi}
\label{D:A:Mem}
We define $P_n^0=(S_n^0\cup\{\otimes\},\SPBKL\nno,\mu_\nn^0,I_\nn^0)$ to be the partial automaton where for all $x\in \SPBKL\nno$, 
 $$
I_n^0(x)=\begin{cases}
       (0,I_\nno(x),\bar\xx)&\text{if $I_\nno(x)\not=\otimes$,}\\
       \otimes &\text{if $I_\nno(x)=\otimes$.}
      \end{cases}
$$
and for all $(0,s,m)\in S_\nn^0$ and for all $x \in \SPBKL\nno$, 
$$
 \mu_n^0((0,s,m),x)=\begin{cases}
              (0,\mu_\nno(s,x),m\cup\bar\xx)&\text{if $\mu_\nno(s,x)\not=\otimes$,}\\
	\otimes&\text{if $\mu_\nno(s,x)=\otimes$.}
           \end{cases}
 $$
\end{defi}
\begin{prop}
\label{P:A:Mem}
The partial automaton $P_\nn^0$ recognizes the language $\Pi\left(R_\nno^\ast\right)$. Moreover an accepted $\SPBKL\nn$-word $\Pi(\ww)$ contains an $\aa\indi\nn$-barrier if and only if $P_{\nn,0}$ has state~$(s,m)$ with $\aa\indi\nn\in m$ after reading $\Pi(\ww)$.
\end{prop}

\begin{proof}
 Let $\ww$ be an $\SPBKL\nn$-word of length $\ell$, $\mathcal{A}$ and $\mathcal{A}'$ be the closure of $P_\nno$ and $P_\nn^0$ respectively. 
 We denote by $s_\kk$ (resp. $(s'_\kk,m_\kk)$) the state of automaton~$\mathcal{A}$ (resp. $\mathcal{A}'$) after reading the $k$-th letter of $\Pi(\ww)$. If $\ww$ does not contains an $\aa\indi\nn$-barrier then $(s'_\kk,m_\kk)$ is equal to $(s_\kk,\emptyset)$ for all $\kk\in[1,\ell]$. Hence $\Pi(\ww)$ is accepted or not by the two automata and in particular $m_\ell$ is the empty set.
 Assume now $\ww$ contains an $\aa\indi\nn$-barrier. 
 Let $\ell'$ be the first occurrence on such a barrier in~$\Pi(\ww)$.
 By construction of $\mu_\nn$ we have $(s'_\kk,m_\kk)$ with $\aa\indi\nn\in m_\kk$ for $\kk\geq\ell'$ except if $s_\kk=\otimes$. 
 As~$\mathcal{A}$ (resp. $\mathcal{A}'$) recognizes the word $\Pi(\ww)$ if and only if $s_\ell$ (resp. $s'_\ell)$ is different from~$\otimes$, the word $\ww$ is recognized or not by both automata. 
 Moreover, in this case $m_\ell$ contains $\aa\indi\nn$.
\end{proof}

As the only $\aa\indi4$-barrier in $\SPBKL4$ is $\aa13$, the partial automaton  $P_4^0$ is obtained from $P_3$ by connecting edges labelled $\aa13$ to a copy of $P_3$, as illustrated on figure~\ref{F:A:Mem1}

\begin{figure}[h!]
\label{F:A:Mem1}
\begin{center}
  \begin{tikzpicture}[x=0.05cm,y=0.05cm,every path/.style={->,>=latex,line width=0.8,label distance=0cm}]
\state{A1}{(240:20)}{0,1}{\emptyset};
\state{A2}{(120:20)}{0,2}{\emptyset};
\state{A0}{(0:20)}{0,3}{\emptyset};
\initstatel{A1}{180}{\aaa23};
\initstatelg{A2}{180}{\aaa13};
\myloop{A1}{\aaa23}{240}{-8};
\myloopg{A2}{}{120}{-9};
\myloop{A0}{\aaa12}{180}{-6};
\path[color=gray] (A1) edge [bend left] node[left=-2pt] {} (A2);
\path (A2) edge [bend left] node[below left=-3pt] {$\scriptstyle\aaa12$}(A0);
\path (A0) edge [bend left] node[below right=-3pt] {$\scriptstyle\aaa23$}(A1);
\begin{scope}[shift={(60,0)}]
\state{B1}{(240:20)}{0,1}{\aaa24};
\state{B2}{(120:20)}{0,2}{\aaa24};
\state{B0}{(0:20)}{0,3}{\aaa24};
\myloop{B1}{\aaa23}{240}{-8};
\myloop{B2}{\aaa13}{120}{-9};
\myloop{B0}{\aaa12}{0}{-6};
\path (B1) edge [bend left] node[right=-2pt] {$\scriptstyle\aaa13$} (B2);
\path (B2) edge [bend left] node[above right=-3pt] {$\scriptstyle\aaa12$}(B0);
\path (B0) edge [bend left] node[below right=-3pt] {$\scriptstyle\aaa23$}(B1);
\end{scope}
\initstateld{B2}{165}{\aaa13};
\draw[dashed] (A2) to [out=20,in=192] node[above=-2pt,pos=0.2]{$\scriptstyle\aaa13$} (B2); 
\draw[dashed] (A1) to [out=-35,in=-130] node[below,pos=0.25]{$\scriptstyle\aaa13$} (35,-10) to [out=50,in=220]  (B2); 
\begin{scope}[shift={(140,-15)}]
  \state{A1}{(0,0)}{0,1}{\emptyset};
  \state{A2B}{(0,30)}{0,2}{\aaa24};
  \state{A0B}{(30,30)}{0,3}{\aaa24};
  \state{A1B}{(30,0)}{0,1}{\aaa24};
  \initstatel{A1}{180}{\aaa23};
  \initstatel{A2B}{180}{\aaa13};
  \myloop{A1}{\aaa23}{225}{-9};
  \myloop{A2B}{\aaa13}{135}{-10};
  \myloop{A0B}{\aaa12}{45}{-10};  
  \myloop{A1B}{\aaa23}{315}{-9};

  \path (A1) edge node[left=-3pt] {$\scriptstyle\aaa13$} (A2B);
  \path (A2B) edge node[above=-3pt] {$\scriptstyle\aaa12$} (A0B);
  \path (A0B) edge node[right=-3pt] {$\scriptstyle\aaa23$} (A1B);
  \path (A1B) edge node[above right=-3pt] {$\scriptstyle\aaa13$} (A2B);
\end{scope}
\end{tikzpicture}
\end{center}
\caption{The partial automaton $P_4^0$. Obsolete transitions from $P_3$ are in gray. New added transitions are dashed. The right partial automaton is $P_4^0$ without inaccessible states.}
\end{figure}
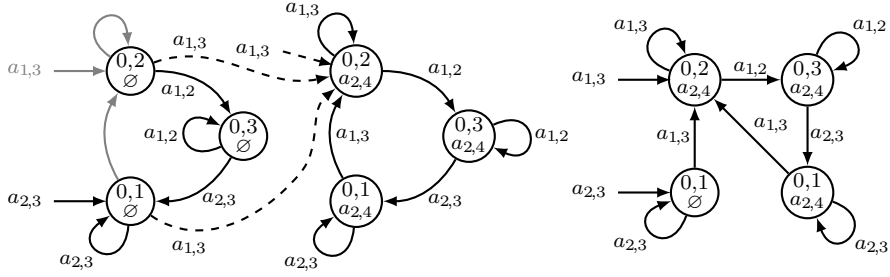

For $t=(0,s,m)\in S_n^0$ we define $\ff\nn^\kk(t)$ to be $(k,s,m)$.
We also define $S_n^k$ to be $\ff\nn^\kk\left(S_n^0\right)$ and  
$$
P_n^k=\left(S_n^\kk,\ff\nn^\kk\left(\SPBKL\nno\right),\mu_\nn^\kk,I_\nn^\kk\right)
$$
to be the partial automaton given by $I_n^k(\ff\nn^\kk(x))=\ff\nn^\kk\left(I_n^0(x)\right)$
and $$\mu_n^k((\kk,s,m),\ff\nn^\kk(x))=\ff\nn^\kk\left(\mu_n^0((0,s,m),x)\right)$$ with the convention $\ff\nn^\kk(\otimes)=\otimes$. In other words, $P_n^k$ is obtained from~$P_n^0$ by replacing the letter $x$ by $\ff\nn^\kk(x)$ and state $(0,s,m)$ by $(k,s,m)$. We obtain immediately that $P_n^k$ recognizes the word $\ff\nn^\kk\left(\Pi(\ww)\right)$ if and only if $P_n^0$ recognizes $\Pi(\ww)$.

We can now construct the partial automaton $P_n$ by plugging together $\nn$ partial automaton $P_n^k$ for $\kk\in[0,\nno]$ together. 

\begin{defi}
 We define $P_n=(S_\nn^\ast\cup\{\otimes\},\SPBKL\nn,\mu_\nn^\ast,I_\nn)$, with $S_\nn^\ast=S_\nn^0\sqcup...\sqcup S_\nn^k$ to be the partial automaton given by
 $$
 I_\nn(x)=\begin{cases}
           I_n^1(\aa\indio\nno) &\text{if $\xx=\aa\indi\nn$ with $\indi\not=1$,}\\
           I_n^2(\aa\nnt\nno) & \text{if $\xx=\aa1\nn$,}\\
           \otimes & \text{otherwise.}
           \end{cases}
 $$
 and with transition function
 $$
 \mu_\nn^\ast((k,s,m),\ff\nn^\kk(\xx))=\begin{cases}
                                   \mu_n^k((k,s,m),\ff\nn^\kk(\xx))&\text{if $\xx\in\SPBKL\nno$,}\\
                                   I_n^\kkp(\ff\nn^\kk(\xx))&\text{if $\xx=\aa\nno\nn$}\\
                                   I_n^\kkp(\ff\nn^\kk(\xx))&\text{if $\xx=\aa\indi\nn$ with $2\leq\indi\leq\nnt$}\\
                                   &\text{and $\aa\indi\nn\in m$,}\\
                                   \otimes&\text{otherwise}
                                  \end{cases}
 $$
 with the convention $I_n^n=I_n^0$.
 \end{defi}

We summarize the construction of the partial automaton $P_n$ on the following diagram. 

$$
 \begin{tikzpicture}[x=0.05cm,y=0.05cm,every path/.style={->,>=latex,line width=0.8,label distance=0cm}]
  \ssstate{P0}{(180:30)}{P_n^0};
  \ssstate{P1}{(120:30)}{P_n^1};
  \ssstate{P2}{(60:30)}{P_n^2};
  \ssstate{Pk}{(0:30)}{P_n^k};
  \ssstate{Pkp}{(300:30)}{P_n^\kkp};
  \ssstate{Pno}{(240:30)}{P_n^\nno};
  \myloop{P0}{\SPBKL\nno}{180}{-6};
  \myloop{P1}{\phi_n\left(\SPBKL{n-1}\right)}{120}{-8};
  \myloop{P2}{\phi_n^2\left(\SPBKL{n-1}\right)}{60}{-8};
  \myloop{Pk}{\phi_n^k\left(\SPBKL{n-1}\right)}{0}{-6};
  \myloop{Pkp}{\phi_n^{k+1}\left(\SPBKL{n-1}\right)}{300}{-8};
  \myloop{Pno}{\phi_n^{n-1}\left(\SPBKL{n-1}\right)}{240}{-8};
  \draw (P0) to [out=75,in=225] node[above left=-2pt] {$\scriptstyle T_0$}(P1);
  \draw (P1) to [out=15,in=165] node[above=-1pt] {$\scriptstyle T_1$}(P2);
  \draw[dashed] (P2) to [out=315,in=105] (Pk);
  \draw (Pk) to [out=255,in=45] node[below right=-2pt] {$\scriptstyle T_k$}(Pkp);
  \draw[dashed] (Pkp) to [out=195,in=345] (Pno);
  \draw (Pno) to [out=135,in=285] node[below left=-2pt] {$\scriptstyle T_\nno$} (P0);
  \draw[-] (0,0) node {$P_\nn$};
  \draw[<-,>=latex,line width=0.8] (P1.center) ++ (180:12pt) --++ (180:20pt) node[left]{$\substack{\aa\indi\nn\\\scriptscriptstyle\indi\not=1}$};
  \draw[<-,>=latex,line width=0.8] (P2.center) ++ (0:12pt) --++ (0:20pt) node[right]{$\scriptstyle\aa1n$};
 \end{tikzpicture}
 $$
 An arrow labelled $T_\kk$ represents the set of transitions $\mu_\nn^\ast((\kk,s,m,\ff\nn^\kk(\aa\indi\nn))$.

\begin{lemm}
\label{L:Automata}
 The partial automaton $P_\nn$ recognizes the language $\Pi\left(R_\nn^\ast\right)$.
\end{lemm}

\begin{proof}
 Let $\mathcal{A}$ be the closure of $P_\nn$ and $\ww$ be a non empty $\SPBKL\nn$-word. 
 There exists a unique sequence $(\ww_b,...,\ww_1)$ of $\SPBKL\nno$-words such that $\ww_b\not=\varepsilon$, $\ww$ is equal to 
 $$
 \ff\nn^\brdio(\ww_\kk)\cdot...\cdot\ff\nn(\ww_2)\cdot\ww_1
 $$ and for all $\ii$, the word $\ff\nn^\ii(\ww_\ii)$ is the maximal suffix of $\ff\nn^\kko(\ww_\kk)\cdot...\cdot\ff\nn^\ii(\ww_i)$ belonging to $\ff\nn^\ii\left(\SPBKL\nno\right)$.
 By definition of $I_n$, the word $\Pi(w)$ is accepted by~$P_\nn$ only if $\ww$ ends by a letter $\aa\indi\nn$ for some $\indi$. 
 We assume now that $\ww$ is such a word. 
 Thus the first integer $\jj$ such that $\ww_\jj$ is non empty is $2$ or $3$. 
 More precisely, we have $\jj=2$ if $\indi>1$ and $\jj=3$ if $\indi=1$ holds. 
 In both cases, the reading of $\Pi(\ww)$ starts by a state coming form $P_\nn^\jj$. 
 The automaton reaches a state different from one of $P_\nn^\jj$ if it goes to the state $\otimes$ or if it reads a letter outside of $\ff\nn^\jjo(\SPBKL\nn)$, \ie, a letter of $\ff\nn^\jj(\ww_\jjp)$. This is a general principle : after reading a letter of $\ff\nn^\iio(\ww_\ii)$ the automaton $\mathcal{A}$ is in state $(t,s,m)$ with $t=\ii\mod\nn$. By construction of $P_\nn^\tt$, the word $\ff\nn^\iio(\ww_\ii)$ provides an accepted state if and only if $\ww_\ii$ is a word of $\Pi(R_\nno)$. At this point we have shown that $\Pi(\ww)$ is accepted by $\mathcal{A}$ only if $\ww$ is empty or if $\ww$ satisfies $\last\www=\aa\indi\nn$ together with Conditions $(i)$, $(ii)$ and $(iii)$ of Corollary~\ref{C:Main}.
 Let $\ii$ be in $[\jj,\kko]$, and assume that $\mathcal{A}$ is in an acceptable state $(t,s,m)$ with $t=i\mod n$ after reading the word $\Pi(\ff\nn^\iio(\ww_\ii)\cdot...\cdot\ff\nn(\ww_2)\cdot\ww_1)$. 
 We denote by $\xx$ the letter~$\last{\ww}_\iip$.
 By construction of $\ww_\iip$ we have $\xx\not\in\ff\nn^\iio(\SPBKL\nno)$ and so $\xx=\ff\nn^\ii(\aa\indi\nn)$ for some $\indi$.
By definition of $\mu_\nn^\ast$ we have $\mu_\nn^\ast((t,s,m),\ff\nn^\ii(\aa\indi\nn))\not=\otimes$ if and only if $\indi=\nno$ of $\indi\in[2,\nnt]$ and $\aa\indi\nn\in m$. By construction of $P_\nn^\tt$, we have $\aa\indi\nn\in m$ if and only if $\ww_\ii$ contains an $\aa\indi\nn$, which corresponds to Condition $(iv)$ of Corollary~\ref{C:Main}.
Eventually, by Corollary~\ref{C:Main}, the word $\Pi(\ww)$ is accepted by $\mathcal{A}$ if and only if $\ww\in R_\nn^\ast$.
\end{proof}
 
 \begin{figure}[h!]
\begin{tikzpicture}[x=0.05cm,y=0.05cm,every path/.style={->,>=latex,line width=0.8,label distance=0cm}]
  \state{A1}{(0,0)}{1,1}{\emptyset};
  \state{A2B}{(0,30)}{1,2}{\aaa24};
  \state{A0B}{(30,30)}{1,3}{\aaa24};
  \state{A1B}{(30,0)}{1,1}{\aaa24};
  \myloop{A1}{\aaa34}{225}{-9};
  \myloop{A2B}{\aaa24}{135}{-10};
  \myloops{A0B}{\aaa23}{45}{30}{42};  
  \myloops{A1B}{\aaa34}{315}{30}{-13};
  \path (A1) edge node[left=-3pt] {$\scriptstyle\aaa24$} (A2B);
  \path (A2B) edge node[above=-3pt] {$\scriptstyle\aaa23$} (A0B);
  \path (A0B) edge node[right=-3pt] {$\scriptstyle\aaa34$} (A1B);
  \path (A1B) edge node[above right=-3pt] {$\scriptstyle\aaa24$} (A2B);
  \coordinate (CANB) at (70,0);
  \coordinate (CANB2) at (35,-20);
  \coordinate (CAB) at (70,30);
  \coordinate (CAB2) at (35,50);
 
 \begin{scope}[shift={(65,-65)}]
  \state{B1}{(0,30)}{2,1}{\emptyset};
  \state{B2B}{(30,30)}{2,2}{\aaa24};
  \state{B0B}{(30,0)}{2,3}{\aaa24};
  \state{B1B}{(0,0)}{2,1}{\aaa24};
  \myloop{B1}{\aaa14}{125}{-9};
  \myloop{B2B}{\aaa13}{45}{-10};
  \myloops{B0B}{\aaa34}{315}{42}{0};  
  \myloops{B1B}{\aaa14}{225}{-12}{0};
  \path (B1) edge node[above=-3pt] {$\scriptstyle\aaa13$} (B2B);
  \path (B2B) edge node[right=-3pt] {$\scriptstyle\aaa34$} (B0B);
  \path (B0B) edge node[below=-1pt] {$\scriptstyle\aaa14$} (B1B);
  \path (B1B) edge node[above left=-3pt] {$\scriptstyle\aaa13$} (B2B);
  \coordinate (CBNB) at (0,-40);
  \coordinate (CBNB2) at (-20,-5);
  \coordinate (CBB) at (30,-40);
  \coordinate (CBB2) at (50,-5);
 \end{scope}
 \begin{scope}[shift={(3,-130)}]
  \state{C1}{(30,30)}{3,1}{\emptyset};
  \state{C2B}{(30,0)}{3,2}{\aaa24};
  \state{C0B}{(0,0)}{3,3}{\aaa24};
  \state{C1B}{(0,30)}{3,1}{\aaa24};
  \myloop{C1}{\aaa12}{45}{-11};
  \myloop{C2B}{\aaa24}{315}{-10};
  \myloops{C0B}{\aaa14}{225}{0}{-13};  
  \myloops{C1B}{\aaa12}{135}{0}{43};
  \path (C1) edge node[right=-3pt] {$\scriptstyle\aaa24$} (C2B);
  \path (C2B) edge node[below=-1pt] {$\scriptstyle\aaa14$} (C0B);
  \path (C0B) edge node[left=-3pt] {$\scriptstyle\aaa12$} (C1B);
  \path (C1B) edge node[below left=-3pt] {$\scriptstyle\aaa24$} (C2B);
  \coordinate (CCNB) at (-40,30);
  \coordinate (CCNB2) at (-5,50);
  \coordinate (CCB) at (-40,0);
  \coordinate (CCB2) at (-5,-20);
 \end{scope}
  \begin{scope}[shift={(-65,-65)}]
  \state{D1}{(30,0)}{0,1}{\emptyset};
  \state{D2B}{(0,0)}{0,2}{\aaa24};
  \state{D0B}{(0,30)}{0,3}{\aaa24};
  \state{D1B}{(30,30)}{0,1}{\aaa24};
  \myloop{D1}{\aaa23}{315}{-9};
  \myloop{D2B}{\aaa13}{225}{-10};
  \myloops{D0B}{\aaa12}{135}{-12}{29};  
  \myloops{D1B}{\aaa23}{45}{42}{29};
  \path (D1) edge node[below=-1pt] {$\scriptstyle\aaa13$} (D2B);
  \path (D2B) edge node[left=-3pt] {$\scriptstyle\aaa12$} (D0B);
  \path (D0B) edge node[above=-3pt] {$\scriptstyle\aaa23$} (D1B);
  \path (D1B) edge node[below right=-3pt] {$\scriptstyle\aaa13$} (D2B);
  \coordinate (CDNB) at (30,70);
  \coordinate (CDNB2) at (50,35);
  \coordinate (CDB) at (0,70);
  \coordinate (CDB2) at (-20,35);
 \end{scope}

 \draw[-] (A1) to [out=315,in=180] (CANB2);
 \draw[-] (A2B) to [out=295,in=180] (CANB2);
 \draw[-] (CANB2) to [out=0,in=180] (CANB);
 \draw[-] (A1B) to (CANB);
 \draw[-] (A0B) to [out=315,in=180] (CANB);
 \draw(CANB)  to [out=0,in=45,looseness=1.6] node[below,pos=0]{$\scriptstyle\aaa14$} (B1);

 \draw[-] (A2B) to [out=45,in=180] (CAB2) to [out=0,in=180] (CAB);
 \draw[-] (A0B) to (CAB);
 \draw[-] (A1B) to [out=45,in=180] (CAB);
 \draw(CAB) to [out=0,in=90,looseness=1.3] node[above=-2pt,pos=0]{$\scriptstyle\aaa13$} (B2B);
 
 \draw[-] (B1) to [out=225,in=90] (CBNB2);
 \draw[-] (B2B) to [out=205,in=90] (CBNB2);
 \draw[-] (CBNB2) to [out=270,in=90] (CBNB);
 \draw[-] (B1B) to (CBNB);
 \draw[-] (B0B) to [out=225,in=90] (CBNB);
 \draw(CBNB)  to [out=270,in=315,looseness=1.6] node[left=-2,pos=0]{$\scriptstyle\aaa12$} (C1);
 
 \draw[-] (B2B) to [out=315,in=90] (CBB2) to [out=270,in=90] (CBB);
 \draw[-] (B0B) to (CBB);
 \draw[-] (B1B) to [out=315,in=90] (CBB);
 \draw(CBB) to [out=270,in=0,looseness=1.3] node[right=-1pt,pos=0]{$\scriptstyle\aaa24$} (C2B);
 
 \draw[-] (C1) to [out=125,in=0] (CCNB2);
 \draw[-] (C2B) to [out=115,in=0] (CCNB2);
 \draw[-] (CCNB2) to [out=180,in=0] (CCNB);
 \draw[-] (C1B) to (CCNB);
 \draw[-] (C0B) to [out=135,in=0] (CCNB);
 \draw(CCNB) to [out=180,in=225,looseness=1.6] node[above,pos=0]{$\scriptstyle\aaa23$} (D1);
 
 \draw[-] (C2B) to [out=225,in=0] (CCB2) to [out=180,in=0] (CCB);
 \draw[-] (C0B) to (CCB);
 \draw[-] (C1B) to [out=225,in=0] (CCB);
 \draw(CCB) to [out=180,in=270,looseness=1.3] node[below,pos=0]{$\scriptstyle\aaa13$} (D2B);
 
 \draw[-] (D1) to [out=35,in=270] (CDNB2);
 \draw[-] (D2B) to [out=25,in=270] (CDNB2);
 \draw[-] (CDNB2) to [out=90,in=270] (CDNB);
 \draw[-] (D1B) to (CDNB);
 \draw[-] (D0B) to [out=25,in=270] (CDNB);
 \draw(CDNB) to [out=90,in=135,looseness=1.6] node[right=-2,pos=0]{$\scriptstyle\aaa34$} (A1);
 
 \draw[-] (D2B) to [out=135,in=270] (CDB2) to [out=90,in=270] (CDB);
 \draw[-] (D0B) to (CDB);
 \draw[-] (D1B) to [out=135,in=270] (CDB);
 \draw(CDB) to [out=90,in=180,looseness=1.3] node[left=-2,pos=0]{$\scriptstyle\aaa24$} (A2B);
 \initstatea{A2B}{90}{\aaa24};
 \initstateb{A1}{270}{\aaa34};
 \initstatel{B1}{180}{\aaa14}
 \end{tikzpicture}
 \caption{Partial automaton recognizing the language $\Pi\left(R_4^\ast\right)$.}
 \end{figure}
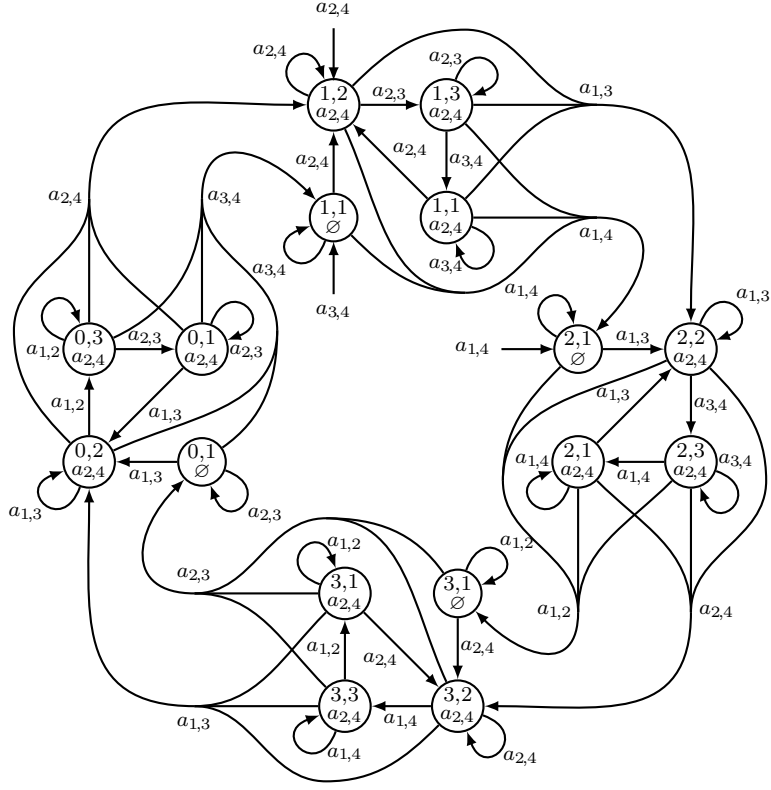

Assume that an automaton $\mathcal{A}_{\nno}=(S_\nno\cup\{\otimes\},\SPBKL\nno,\mu_\nno,i)$ recognizing the language~$\Pi(R_\nno)$ for $\nn\geq4$ is given. Using the partial automaton $P_\nn=(S_\nn^\ast\cup\{\otimes\},\SPBKL\nn,\mu_\nn^\ast,I_\nn)$ we construct the automaton $\mathcal{A}_\nn=(S_\nn\cup\{\otimes\},\SPBKL\nn,\mu_\nn,i)$ defined by $S_\nn=S_\nno\sqcup S_\nn^\ast$ and
$$
\mu_\nn(s,x)=\begin{cases}
              \mu_\nno(s,x)&\text{if $s\in S_\nno$ and $x\in\SPBKL\nno$,}\\
              I_\nn(x)&\text{if $s\in S_\nno$ and $x\in\SPBKL\nn\setminus\SPBKL\nno$,}\\
             \mu_\nn^\ast(s,x) & \text{if $s\in S_\nn^\ast$.}
             \end{cases}
$$

\begin{prop}
\label{P:Automata}
If $\mathcal{A}_\nno$ recognizes $\Pi(R_\nno)$, the automaton $\mathcal{A}_\nn$ recognizes the language $\Pi(R_\nn)$.
\end{prop}

\begin{proof}
Let $\ww$ be an $\SPBKL\nn$-word, $\ww_1$ be the maximal suffix of $\ww$ which is an $\SPBKL\nno$-word and $\ww'$ be the corresponding prefix. By Corollary~\ref{C:Main}, the word $\ww$ is rotating if and only if $\ww_1$ and $\ww'$ are.
By construction of $\mathcal{A}_\nn$, the automaton is in acceptable state after reading $\Pi(\ww_1)$ if and only if $\ww_1$ is an $(\nno)$-rotating word. Hence $\ww$ is accepted only if $\ww_1$ is rotating. Assume that it is the case.
By Lemma~\ref{L:Automata} the automaton~$\mathcal{A}_\nn$ is always in an acceptable state after reading $\Pi(\ww')$ if and only if the word~$\ww'$ is rotating. Eventually the word $\Pi(\ww)$ is accepted by $\mathcal{A}$ if and only if $\ww_1$ and $\ww'$ are both rotating, which is equivalent to $\ww$ is rotating.
\end{proof}

By Proposition~\ref{P:Automata}, the language $\Pi(R_\nn)$ is regular and so we obtain:

\begin{thrm}
 The language of $\nn$-rotating words $R_n$ is regular.
\end{thrm}

 \begin{figure}[t!]
\begin{tikzpicture}[x=0.05cm,y=0.05cm,every path/.style={->,>=latex,line width=0.8,label distance=0cm}]
\begin{scope}[shift={(15,85)}]
  \sstate{3}{(270:20)}{3};
\sstate{1}{(150:20)}{1};
\sstate{2}{(30:20)}{2};
\sstate{0}{(0,35)}{0};
\path (0) edge node[left,pos=0.2] {$\scriptstyle\aaa23$} (1);
\path (0) edge node[right,pos=0.2] {$\scriptstyle\aaa13$} (2);
\path (1) edge [bend left] node[above=-1pt] {$\scriptstyle\aaa13$} (2);
\path (2) edge [bend left] node[below right=-3pt] {$\scriptstyle\aaa12$} (3);
\path (3) edge [bend left] node[above right=-2pt,pos=0.8] {$\scriptstyle\aaa23$} (1);
\myloop{0}{\aaa12}{45}{-9};
\myloop{2}{\aaa13}{45}{-9};
\myloops{3}{\aaa12}{270}{9}{-28};
\myloops{1}{\aaa23}{135}{-25}{25};
\initstate{0}{135};
\coordinate (CT0) at (-37,20);
  \coordinate (CT1) at (-15,-30);
\coordinate (CT2) at (-35,-35);
\coordinate (CT3) at (-42,20);
\coordinate (CT4) at (45,-30);
  \end{scope}
  
  \state{A1}{(0,0)}{1,1}{\emptyset};
  \state{A2B}{(0,30)}{1,2}{\aaa24};
  \state{A0B}{(30,30)}{1,3}{\aaa24};
  \state{A1B}{(30,0)}{1,1}{\aaa24};
  \myloop{A1}{\aaa34}{225}{-9};
  \myloops{A2B}{\aaa24}{135}{-8}{46};
  \myloops{A0B}{\aaa23}{45}{30}{42};  
  \myloops{A1B}{\aaa34}{315}{30}{-13};
  \path (A1) edge node[left=-2pt,pos=0.7] {$\scriptstyle\aaa24$} (A2B);
  \path (A2B) edge node[above=-3pt] {$\scriptstyle\aaa23$} (A0B);
  \path (A0B) edge node[right=-3pt] {$\scriptstyle\aaa34$} (A1B);
  \path (A1B) edge node[above right=-3pt] {$\scriptstyle\aaa24$} (A2B);
  \coordinate (CANB) at (70,0);
  \coordinate (CANB2) at (35,-20);
  \coordinate (CAB) at (70,30);
  \coordinate (CAB2) at (35,50);
 
 \begin{scope}[shift={(65,-65)}]
  \state{B1}{(0,30)}{2,1}{\emptyset};
  \state{B2B}{(30,30)}{2,2}{\aaa24};
  \state{B0B}{(30,0)}{2,3}{\aaa24};
  \state{B1B}{(0,0)}{2,1}{\aaa24};
  \myloop{B1}{\aaa14}{125}{-9};
  \myloop{B2B}{\aaa13}{45}{-10};
  \myloops{B0B}{\aaa34}{315}{42}{0};  
  \myloops{B1B}{\aaa14}{225}{-12}{0};
  \path (B1) edge node[above=-3pt] {$\scriptstyle\aaa13$} (B2B);
  \path (B2B) edge node[right=-3pt] {$\scriptstyle\aaa34$} (B0B);
  \path (B0B) edge node[above=-3pt] {$\scriptstyle\aaa14$} (B1B);
  \path (B1B) edge node[above left=-3pt] {$\scriptstyle\aaa13$} (B2B);
  \coordinate (CBNB) at (0,-40);
  \coordinate (CBNB2) at (-20,-5);
  \coordinate (CBB) at (30,-40);
  \coordinate (CBB2) at (50,-5);
 \end{scope}
 \begin{scope}[shift={(3,-130)}]
  \state{C1}{(30,30)}{3,1}{\emptyset};
  \state{C2B}{(30,0)}{3,2}{\aaa24};
  \state{C0B}{(0,0)}{3,3}{\aaa24};
  \state{C1B}{(0,30)}{3,1}{\aaa24};
  \myloop{C1}{\aaa12}{45}{-11};
  \myloop{C2B}{\aaa24}{315}{-10};
  \myloops{C0B}{\aaa14}{225}{0}{-13};  
  \myloops{C1B}{\aaa12}{135}{0}{43};
  \path (C1) edge node[right=-3pt] {$\scriptstyle\aaa24$} (C2B);
  \path (C2B) edge node[below=-1pt] {$\scriptstyle\aaa14$} (C0B);
  \path (C0B) edge node[left=-3pt] {$\scriptstyle\aaa12$} (C1B);
  \path (C1B) edge node[below left=-3pt] {$\scriptstyle\aaa24$} (C2B);
  \coordinate (CCNB) at (-40,30);
  \coordinate (CCNB2) at (-5,50);
  \coordinate (CCB) at (-40,0);
  \coordinate (CCB2) at (-5,-20);
 \end{scope}
  \begin{scope}[shift={(-65,-65)}]
  \state{D1}{(30,0)}{0,1}{\emptyset};
  \state{D2B}{(0,0)}{0,2}{\aaa24};
  \state{D0B}{(0,30)}{0,3}{\aaa24};
  \state{D1B}{(30,30)}{0,1}{\aaa24};
  \myloop{D1}{\aaa23}{315}{-9};
  \myloop{D2B}{\aaa13}{225}{-10};
  \myloops{D0B}{\aaa12}{135}{-12}{29};  
  \myloops{D1B}{\aaa23}{45}{42}{29};
  \path (D1) edge node[below=-1pt] {$\scriptstyle\aaa13$} (D2B);
  \path (D2B) edge node[left=-3pt] {$\scriptstyle\aaa12$} (D0B);
  \path (D0B) edge node[below=-1pt] {$\scriptstyle\aaa23$} (D1B);
  \path (D1B) edge node[below right=-3pt] {$\scriptstyle\aaa13$} (D2B);
  \coordinate (CDNB) at (30,70);
  \coordinate (CDNB2) at (50,35);
  \coordinate (CDB) at (0,70);
  \coordinate (CDB2) at (-20,35);
 \end{scope}

 \draw[-] (A1) to [out=315,in=180] (CANB2);
 \draw[-] (A2B) to [out=295,in=180] (CANB2);
 \draw[-] (CANB2) to [out=0,in=180] (CANB);
 \draw[-] (A1B) to (CANB);
 \draw[-] (A0B) to [out=315,in=180] (CANB);
 \draw(CANB)  to [out=0,in=55,looseness=1.3] node[above right=-4,pos=0.2]{$\scriptstyle\aaa14$} (B1);

 \draw[-] (A2B) to [out=45,in=180] (CAB2) to [out=0,in=180] (CAB);
 \draw[-] (A0B) to (CAB);
 \draw[-] (A1B) to [out=45,in=180] (CAB);
 \draw(CAB) to [out=0,in=90,looseness=1.3] node[above=-2pt,pos=0.1]{$\scriptstyle\aaa13$} (B2B);
 
 \draw[-] (B1) to [out=225,in=90] (CBNB2);
 \draw[-] (B2B) to [out=205,in=90] (CBNB2);
 \draw[-] (CBNB2) to [out=270,in=90] (CBNB);
 \draw[-] (B1B) to (CBNB);
 \draw[-] (B0B) to [out=225,in=90] (CBNB);
 \draw(CBNB)  to [out=270,in=325,looseness=1.3] node[left=-2,pos=0]{$\scriptstyle\aaa12$} (C1);
 
 \draw[-] (B2B) to [out=315,in=90] (CBB2) to [out=270,in=90] (CBB);
 \draw[-] (B0B) to (CBB);
 \draw[-] (B1B) to [out=315,in=90] (CBB);
 \draw(CBB) to [out=270,in=0,looseness=1.3] node[right=-1pt,pos=0]{$\scriptstyle\aaa24$} (C2B);
 
 \draw[-] (C1) to [out=125,in=0] (CCNB2);
 \draw[-] (C2B) to [out=115,in=0] (CCNB2);
 \draw[-] (CCNB2) to [out=180,in=0] (CCNB);
 \draw[-] (C1B) to (CCNB);
 \draw[-] (C0B) to [out=135,in=0] (CCNB);
 \draw(CCNB) to [out=180,in=235,looseness=1.3] node[above,pos=0]{$\scriptstyle\aaa23$} (D1);
 
 \draw[-] (C2B) to [out=225,in=0] (CCB2) to [out=180,in=0] (CCB);
 \draw[-] (C0B) to (CCB);
 \draw[-] (C1B) to [out=225,in=0] (CCB);
 \draw(CCB) to [out=180,in=270,looseness=1.3] node[below,pos=0]{$\scriptstyle\aaa13$} (D2B);
 
 \draw[-] (D1) to [out=35,in=270] (CDNB2);
 \draw[-] (D2B) to [out=25,in=270] (CDNB2);
 \draw[-] (CDNB2) to [out=90,in=270] (CDNB);
 \draw[-] (D1B) to (CDNB);
 \draw[-] (D0B) to [out=25,in=270] (CDNB);
 \draw(CDNB) to [out=90,in=145,looseness=1.3] node[right=-2,pos=0]{$\scriptstyle\aaa34$} (A1);
 
 \draw[-] (D2B) to [out=135,in=270] (CDB2) to [out=90,in=270] (CDB);
 \draw[-] (D0B) to (CDB);
 \draw[-] (D1B) to [out=135,in=270] (CDB);
 \draw(CDB) to [out=90,in=180,looseness=1.3] node[left=-2,pos=0]{$\scriptstyle\aaa24$} (A2B);

 \draw[-] (3) to [out=200,in=90] (CT1);
 \draw[-] (2) to [out=240,in=90] (CT1);
 \draw[-] (1) to [out=240,in=90] (CT1);
 \draw[-] (0) to [out=190,in=90] (CT0);
 \draw[-] (CT0) to [out=270,in=90] (CT1);
 \draw(CT1) to node[left=-2,pos=0]{$\scriptstyle\aaa24$}(A2B);
 
 \draw[-] (3) to [out=170,in=90] (CT2);
 \draw[-] (2) to [out=225,in=90] (CT2);
 \draw[-] (1) to [out=225,in=90] (CT2);
 \draw[-] (0) to [out=180,in=90] (CT3) to [out=270,in=90] (CT2);
 \draw(CT2) to [out=270,in=125,looseness=1.5] node[left=-2,pos=0]{$\scriptstyle\aaa34$}(A1);
 
 \draw[-] (2) to [out=315,in=120] (CT4);
 \draw[-] (3) to [out=10,in=120] (CT4);
 \draw[-] (1) to [out=340,in=120] (CT4);
 \draw[-] (0) to [out=0,in=120,looseness=1.7] (CT4);
 \draw(CT4) to [out=300,in=75] node[right=-2,pos=0]{$\scriptstyle\aaa14$} (B1);
 \end{tikzpicture}
 \caption{Automaton $\mathcal{A}_4$ for the language $\Pi(R_4)$.}
 \end{figure}
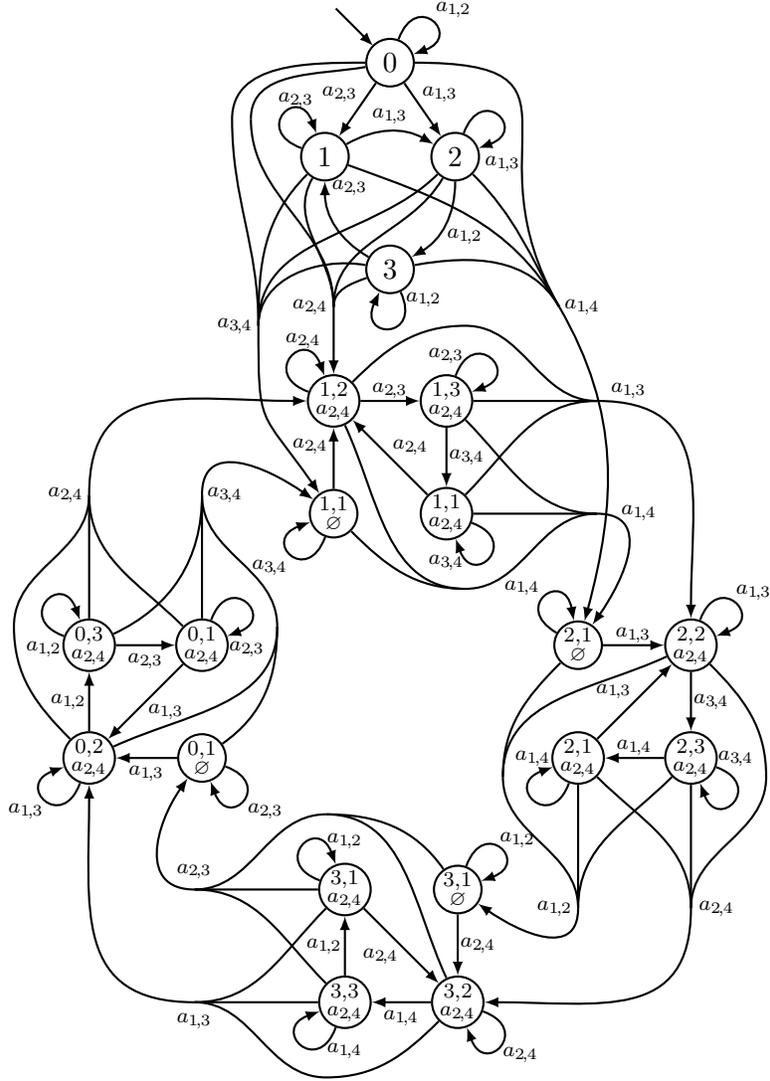
 
 \section*{Further work}

Using syntactical characterization of rotating words we have proved that the language of $\nn$-rotating words is regular.
For $W$ a finite state automaton, we denote by $L(W)$ the language recognized by $W$.
Following \cite{Campbell} and \cite{Epstein1992} we have the following definition:

\begin{defi}
Let $M$ be a monoid. 
A \emph{right automatic structure}, resp. \emph{left automatic structure}, on $M$ consists of a set $A$ of generators of $M$, a finite state automaon $W$ over $A$, and finite state automata~$M_x$ over $(A,A)$, for $x\in A\cup\{\varepsilon\}$, satisfying the following conditions:

$(i)$ the map $\pi:L(W)\to M$ is surjective.

$(ii)$ for $x\in A\cup\{\varepsilon\}$, we have $(u,v)\in L(M_x)$ if and only if $\overline{ux}=\overline{y}$, resp.  $\overline{xu}=\overline{y}$, and both $\uu$ and $\vv$ are elements of $L(W)$.

\end{defi}

Naturally we can ask if the rotating normal form provides an left or right automatic structure for the dual braid monoid $\BKL\nn$. Such a result may needs to obtain some syntactical properties on the word $\xx\,\ww$ or $\ww\,\xx$ where $\ww$ is an $\nn$-rotating word and $\xx$ is an $\SPBKL\nn$-generator. 
At this time no result have been obtained in this direction.
\bibliographystyle{ams-pln}
\bibliography{biblio}

\end{document}